\documentclass[10pt]{amsart}
\usepackage[centertags]{amsmath}
\usepackage{amsfonts}
\usepackage{amssymb}
\usepackage{amsthm}
\usepackage{newlfont}
\usepackage{amscd}
\usepackage{mathrsfs}
\usepackage[all,cmtip]{xy}
\usepackage{tikz}
\usetikzlibrary{trees}
\usepackage[pagebackref=false,colorlinks]{hyperref}
\hypersetup{pdffitwindow=true,linkcolor=purple,citecolor=teal,urlcolor=violet}
\usepackage{cite}
\usepackage{fancyhdr}
\usepackage{stmaryrd}

\usepackage[OT2,OT1]{fontenc}
\newcommand\cyr{%
\renewcommand\rmdefault{wncyr}%
\renewcommand\sfdefault{wncyss}%
\renewcommand\encodingdefault{OT2}%
\normalfont
\selectfont}
\DeclareTextFontCommand{\textcyr}{\cyr}

\newcommand{\mint}{{\times}\kern-0.89em{\int}} 

\usepackage[toc, page] {appendix}

\numberwithin{equation}{section}

\newtheorem{thm}{Theorem}[section]
\newtheorem{cor}[thm]{Corollary}
\newtheorem{lem}[thm]{Lemma}
\newtheorem{prop}[thm]{Proposition}

\newtheorem{choice}[thm]{Choice}
\theoremstyle{definition}
\newtheorem{defn}[thm]{Definition}
\newtheorem{rem}[thm]{Remark}
\newtheorem{assu}[thm]{Assumption}

\newtheorem{ques}[thm]{Question}

\newtheorem{notation}[thm]{Notation}
 
\newcommand{\domain}{B^\times \backslash \widehat{B}^\times / \widehat{R}^{(p), \times}}

\begin{document}
\title{Overconvergent quaternionic forms and anticyclotomic $p$-adic $L$-functions}
\author{Chan-Ho Kim}
\address{School of Mathematics, KIAS, Seoul, Korea}
\email{chanho.math@gmail.com}
\date{\today}
\maketitle
\begin{abstract}
We reinterpret the explicit construction of Gross points given by Chida-Hsieh as a non-Archimedian analogue of the standard geodesic cycle $(i\infty) - (0)$ on the Poincar\'{e} upper half plane. This analogy allows us to consider certain distributions, which can be regarded as anticyclotomic $p$-adic $L$-functions for modular forms of non-critical slope following the overconvergent strategy \`{a} la Stevens. We also give a geometric interpretation of their Gross points for the case of weight two forms. Our construction generalizes those of Bertolini-Darmon, Bertolini-Darmon-Iovita-Spiess, and Chida-Hsieh.
\end{abstract}
\setcounter{tocdepth}{1}
\tableofcontents

\section{Introduction}

\subsection{Overview}
It is now widely known that ``overconvergent methods" yield simpler and more algorithmically efficient constructions of $p$-adic $L$-functions \cite{glenn-oms}, \cite{pollack-stevens} and conjectural algebraic points of elliptic curves, so called Darmon-Stark-Heegner points \cite{darmon-pollack}, \cite{guitart-masdeu}. In \cite{glenn-oms}, Stevens provides a simple and beautiful construction of Mazur-Tate-Teitelbaum $p$-adic $L$-functions of modular forms under cyclotomic extensions using distribution-valued modular symbols. 
In \cite{greenberg-stevens}, measure-valued modular symbols, which can be regarded as a special case, are used in the proof of the exceptional zero conjecture \cite[$\S$15, 16]{mtt} as an essential ingredient.

In this article, we apply Stevens' ``overconvergent'' idea to the anticyclotomic setting.
Instead of using modular symbols, we use automorphic forms on a definite quaternion algebra (quaternionic forms, for short).
Although modular symbols and quaternionic forms have certain similarities in their shape, their domains are fundamentally different. Modular symbols are essentially defined on the upper half plane, which lies in the complex world, and quaternionic forms are defined on the double coset space arising from the quaternion algebra. Note that, in the case of weight two forms, it can be realized in terms of the Bruhat-Tits tree for $\mathrm{PGL}_2(\mathbb{Q}_p)$ or its variant, which lies in the $p$-adic world.

Using the theory of overconvergent modular symbols, it is proved that the evaluation of the overconvergent modular symbol attached to a non-critical slope eigenform at the cycle $(i \infty) - (0)$ on the upper half plane gives us the $p$-adic distribution corresponding to the Mazur-Tate-Teitelbaum $p$-adic $L$-function of the form.

We develop an analogous theory for overconvergent quaternionic forms. Since the domain is fundamentally different from the case of modular symbols, we naturally meet the following question.
\begin{ques} \label{ques:the_basic_question}
What is an analogue of the geodesic cycle $(i \infty) - (0) \in \mathrm{Div}^0(\mathbb{P}^1(\mathbb{Q}))$ in the quaternionic setting?
\end{ques}

The main contribution of this article is to provide an answer to this question by taking the full advantage of the \emph{explicit} construction of Gross points \`{a} la Chida-Hsieh. We will call such an analogue \emph{the explicit Gross point}.

Also, in the case of weight two forms, we give another interpretation of these points in terms of the Bruhat-Tits tree for $\mathrm{PGL}_2(\mathbb{Q}_p)$.
As an application, we are able to generalize the construction of anticyclotomic $p$-adic $L$-functions to modular forms of non-critical slope. Our construction generalizes those in \cite{bertolini-darmon-mumford-tate-1996}, \cite{bertolini-darmon-iovita-spiess-2002}, and \cite{chida-hsieh-p-adic-L-functions}.

In order to do this, we recall the notion of overconvergent quaternionic forms and (re)prove the control theorem for overconvergent quaternionic forms of non-critical slope. This generalizes \cite[$\S$3]{longo-vigni-control}, which deals with the control theorem for the slope zero subspace. Also, our approach yields a certain integrality of the control theorem for the slope zero subspace.

We expect that the explicit Gross points can be reinterpreted as a functional on the completed cohomology for quaternion algebras sending cuspidal eigenforms to (an half of) their anticyclotomic $p$-adic $L$-functions.

In the sequel paper in preparation, we construct \emph{integral} anticyclotomic $p$-adic $L$-functions for Hida families, which are two variable ones, and prove the vanishing of $\mu$-invariant of each member of the families under mild assumptions, generalizing \cite{castella-kim-longo}. In \cite{castella-kim-longo}, a different approach was taken following \cite{longo-vigni-manuscripta} and \cite{castella-longo} using compatible families of Gross points in the tower of Gross curves, so called \emph{big Gross points}. Note that the approach using big Gross points does not work for the non-ordinary case.

The following diagram describes the flowchart for the classical constructions of cyclotomic and anticyclotomic $p$-adic $L$-functions of modular forms. 
The upper(=cyclotomic) part of the diagram is well-documented in \cite{pollack-oms}.
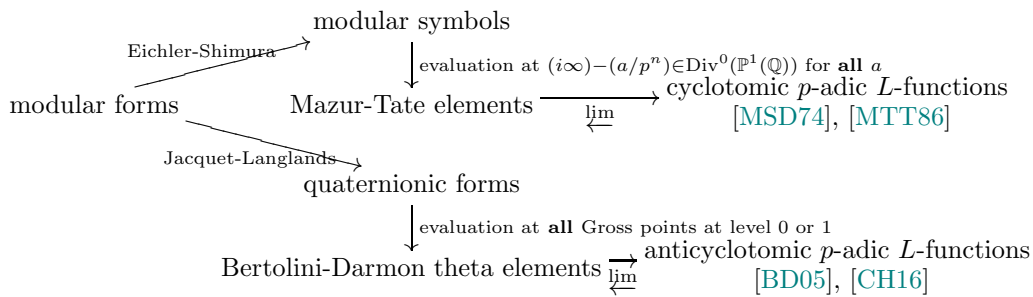
\[
\xymatrix@=1em{
  & \txt{modular symbols} \ar[d]^-{\textrm{evaluation at } (i\infty) - (a/p^n) \in \mathrm{Div}^0(\mathbb{P}^1(\mathbb{Q})) \textrm{ for } \textbf{all } a}\\
\txt{modular forms} \ar[ur]^-{\textrm{Eichler-Shimura}} \ar[dr]_-{\textrm{Jacquet-Langlands}} & \txt{Mazur-Tate elements} \ar[r]_-{\varprojlim} & \txt{cyclotomic $p$-adic $L$-functions\\\cite{msd}, \cite{mtt}} \\
 &  \txt{quaternionic forms} \ar[d]^-{\textrm{evaluation at} \textbf{ all } \textrm{Gross points at level 0 or 1}}  \\
 & \txt{Bertolini-Darmon theta elements} \ar[r]_-{\varprojlim} & \txt{anticyclotomic $p$-adic $L$-functions\\ \cite{bertolini-darmon-imc-2005}, \cite{chida-hsieh-p-adic-L-functions}}
}
\]

The overconvergent method shows that it suffices to evaluate overconvergent modular symbols or overconvergent quaternionic forms ``at one point". This is because
the overconvergent method pushes the complexity of the evaluation of classical quaternionic forms at all Gross points ($\S$\ref{subsec:classical_gross_pts}) into the complexity of the coefficient modules (the distribution modules) of overconvergent quaternionic forms ($\S$\ref{subsec:distributions} and $\S$\ref{subsec:the_distribution}).
The bold part of the following diagram is the main content of this article. 
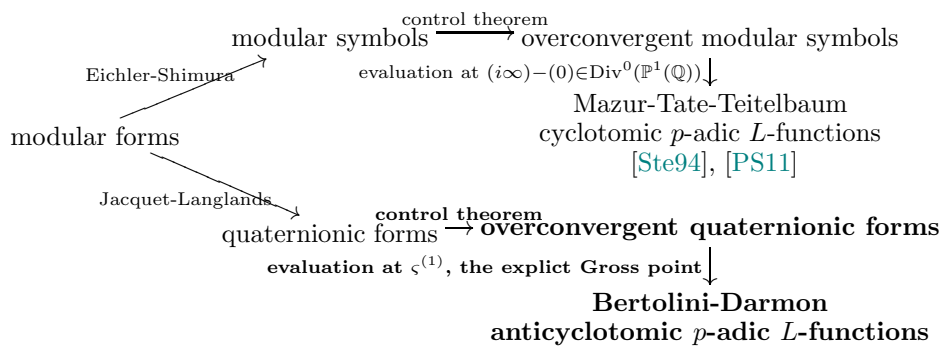
\[
\xymatrix@=1em{
& \txt{modular symbols} \ar[r]^-{\textrm{control theorem}} & \txt{overconvergent modular symbols}   \ar[d]_-{\textrm{evaluation at } (i\infty) - (0) \in \mathrm{Div}^0(\mathbb{P}^1(\mathbb{Q}))}\\
\txt{modular forms} \ar[ur]^-{\textrm{Eichler-Shimura}} \ar[dr]_-{\textrm{Jacquet-Langlands}} & & \txt{Mazur-Tate-Teitelbaum\\cyclotomic $p$-adic $L$-functions\\ \cite{glenn-oms}, \cite{pollack-stevens}} \\
 &  \txt{quaternionic forms} \ar[r]^-{\textbf{control theorem}} & \textbf{overconvergent quaternionic forms}  \ar[d]_-{\textbf{evaluation at $\varsigma^{(1)}$, the explict Gross point}} \\
 &  & \txt{\textbf{Bertolini-Darmon}\\ \textbf{anticyclotomic $p$-adic $L$-functions}}
}
\]

\subsection{Setting the basic stage}
Let $p$ be a prime $\geq$ 3 and $k \in \mathbb{Z}_{\geq 2}$.
Fix an algebraic closure $\overline{\mathbb{Q}}$ of $\mathbb{Q}$ and embeddings $\iota_\infty: \overline{\mathbb{Q}}\hookrightarrow \mathbb{C}$ and $\iota_p: \overline{\mathbb{Q}} \hookrightarrow \mathbb{C}_p$.
Let $\Gamma_0(N)$ be the congruence subgroup of level $N$ with $(N,p) = 1$.
Let $f_k = \sum a_n(f_k)q^n \in S_{k} (\Gamma_0(Np))$ be a $p$-stabilized newform of slope $h = \mathrm{ord}_p (\alpha_p(f_k)) < k-1$ with the convention $\mathrm{ord}_p(p) = 1$, i.e. the slope of $f$ is non-critical.

Fix an imaginary quadratic field $K$ with $(\mathrm{disc}(K), pN) = 1$.
The choice of $K$ determines the decomposition of $N$ as follows:
\begin{equation} \label{eqn:decomposition}
N= N^+ \cdot N^-
\end{equation}
where a prime divisor of $N^+$ splits in $K$ and a prime divisor of $N^-$ is inert in $K$.

\begin{assu} \label{assu:parity}
In Equation (\ref{eqn:decomposition}), $N^-$ is square-free and the product of an \emph{odd} number of primes.
\end{assu}

Let $K_\infty$ be the anticyclotomic $\mathbb{Z}_p$-extension of $K$ and $\Gamma_\infty = \mathrm{Gal}(K_\infty/K) \simeq \mathbb{Z}_p$ (non-canonically). Write $K_n$ for the unique subfield of $K_\infty$ such that $\Gamma_n = \mathrm{Gal}(K_n/K) \simeq \mathbb{Z}/p^n\mathbb{Z}$.

Let $B$ be the definite quaternion algebra over $\mathbb{Q}$ of discriminant $N^-$ and $R = R_{N^+}$ be an (oriented) Eichler order of level $N^+$.
For each prime $\ell \nmid N^-$, we fix an embedding $R_\ell := R \otimes_{\mathbb{Z}} \mathbb{Z}_\ell \hookrightarrow \mathrm{M}_2(\mathbb{Z}_\ell)$ and we identify them under this isomorphism.
Let $\widehat{A} := A \otimes_{\mathbb{Z}} \widehat{\mathbb{Z}}$ for any abelian group $A$.

For each $r \geq 0$, let
$R_{N^+p^r}$ be an Eichler order of level $N^+ p^r$ such that 
$$R^\times_{N^+p^r,p}  := (R_{N^+p^r} \otimes_{\mathbb{Z}} \mathbb{Z}_p )^\times = \lbrace \left( \begin{smallmatrix}
a & b \\ c & d 
\end{smallmatrix} \right) \in \mathrm{GL}_2(\mathbb{Z}_p) : c \in p^r\mathbb{Z}_p \rbrace$$
and its prime-to-$p$ component coincides with that of $\widehat{R}^\times$.
We also write $\Gamma_0(p^r\mathbb{Z}_p) = R^\times_{N^+p^r,p}$.
Note that $\widehat{R}^\times_{N^+p^r}$ corresponds to $\Gamma_0(N^+ p^r)$-level structures in the classical sense.

Let $E$ be a finite extension of $\mathbb{Q}_p$ large enough to contain all the Hecke eigenvalues of $f_k$ and write $\mathcal{O} = \mathcal{O}_E$.

Let $\mathscr{D}_k/\mathbf{D}_k$ be the space of $E$-valued locally/rigid analytic distributions on $\mathbb{Z}_p$ with weight $k$ action of a certain semigroup $\Sigma_0(p)$, respectively. 
Let $\mathscr{D}_k(\mathcal{O})/\mathbf{D}_k(\mathcal{O})$ be the subspaces of $\mathcal{O}$-valued locally/rigid analytic distributions of $\mathscr{D}_k/\mathbf{D}_k$, respectively. See $\S$\ref{subsec:distributions} for detail.

Let $S^{N^-}_{k} (N^+p, E)$ be the space of $E$-valued quaternionic forms of weight $k$, level $N^+ p$, and discriminant $N^-$, and denote its overconvergent variants by $S^{N^-} (N^+p, \mathscr{D}_k)$, $S^{N^-} (N^+p, \mathbf{D}_k)$, which are defined in $\S$\ref{sec:automorphic_forms}.

For any Hecke module $S$, let $S^{(< h)}$ be the subspace of $S$ consisting of the members whose slopes are less than $h$, and $S^{(0)}$ be the slope zero subspace.

\subsection{A precise formulation of Question \ref{ques:the_basic_question} and its answer} \label{subsec:precise_formulation}
For a cuspidal eigenform $f_k$ of non-critical slope, let
 $$\phi^{ms}_{f_k} : \mathrm{Div}^0( \mathbb{P}^1(\mathbb{Q}) ) \to \mathrm{Sym}^{k-2}( \overline{\mathbb{Z}}^2_p )$$
  be the integrally normalized corresponding modular symbol defined using the Eichler-Shimura map.
Looking at the diagram
\[
\xymatrix{
& & \mathrm{SL}_2(\mathbb{R}) \ar@{->>}[d] \\
\mathbb{P}^1(\mathbb{Q}) & &  \mathfrak{h} \ar@{->>}[d] \ar@{-->}[ll]_-{``\textrm{boundary}"} \\
& & \Gamma(1) \backslash \mathfrak{h} & \textrm{the set of homothety classes of lattices in $\mathbb{C}$} ,
}
\]
we may consider $\phi^{ms}_{f_k}$ as ``a function on the upper half plane $\mathfrak{h}$", at least intuitively. Indeed, the modular symbols are computed terms of the period integrals on $\mathfrak{h}$ (as in \cite[$\S$2]{pollack-oms}).
\begin{rem}[on cusps on the adelic formulation]
See \cite[$\S$3]{scholl-modular-units} for the adelic interpretation of cusps, $\mathbb{P}^1(\mathbb{Q})$, on $\mathfrak{h}$.
\end{rem}
By \cite{glenn-oms}, we can uniquely lift $\phi_{f_k}$ to the overconvergent modular symbol
$$\Phi^{ms}_{f_k} : \mathrm{Div}^0( \mathbb{P}^1(\mathbb{Q}) ) \to \mathscr{D}_k .$$
See  also \cite{greenberg-israel} and \cite{pollack-stevens}.
Then the overconvergent modular symbol $\Phi^{ms}_{f_k}$ directly yields the Mazur-Tate-Teitelbaum $p$-adic $L$-function as a distribution by
$$L_p(\mathbb{Q}(\mu_{p^\infty}), f_k) = \Phi^{ms}_{f_k} \left(  (i\infty) - (0)  \right) .$$ 

In the anticyclotomic case, certain special points on the adelic double coset space arising from quaternion algebras, called (classical) Gross points, play the same role as $(i\infty) - (a/p^n) \in \mathrm{Div}^0( \mathbb{P}^1(\mathbb{Q}) )$ for the construction of anticyclotomic $p$-adic $L$-function of weight two ordinary forms. We will review this in $\S$\ref{subsec:classical_gross_pts}.

There are several approaches toward the generalization to the higher weight forms
including \cite{bertolini-darmon-iovita-spiess-2002},  \cite{bertolini-darmon-hida-2007}, and \cite{chida-hsieh-p-adic-L-functions} but with limitations. One of the obstructions is the lack of the ``right infinite level space" where the Gross points live. More precisely, the domain of higher weight quaternionic forms lies in a ``deeper" level than the domain where the Gross points are canonically defined.

Mimicking the above picture in the quaternionic setting, we have a slightly more complicated picture
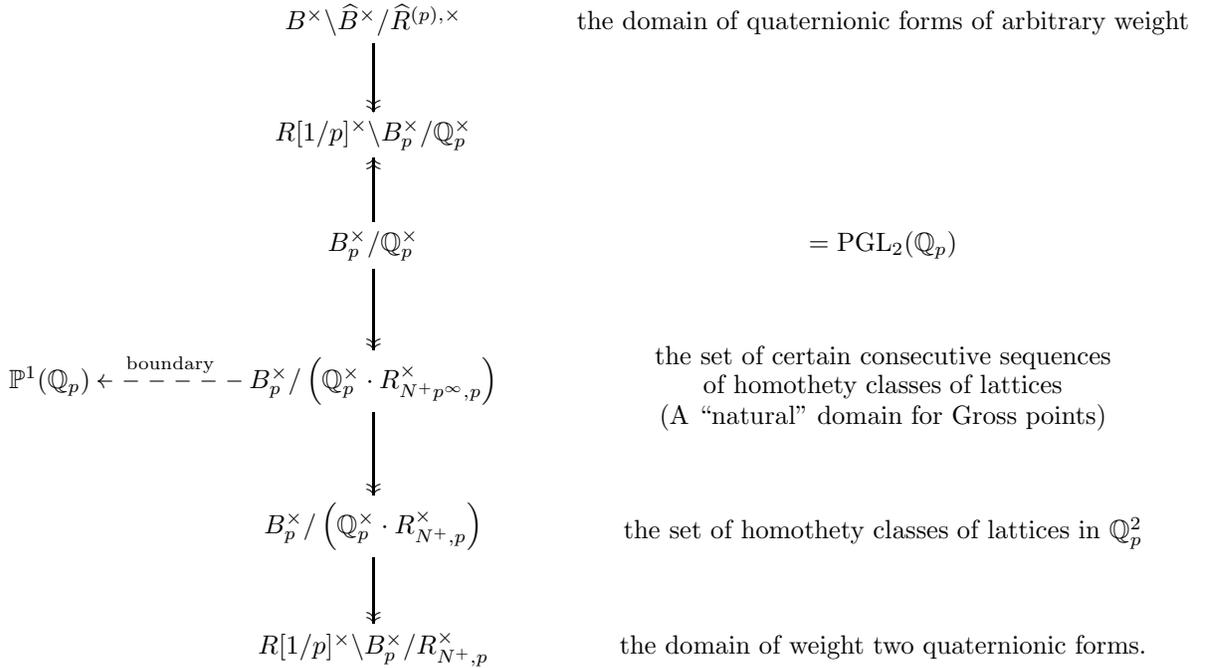
\[
\xymatrix{
& & B^\times \backslash \widehat{B}^\times  / \widehat{R}^{(p), \times} \ar@{->>}[d] & \textrm{the domain of quaternionic forms of arbitrary weight}\\
& & R[1/p]^\times \backslash B^{\times}_p / \mathbb{Q}^\times_p  & \\
& & B^\times_p / \mathbb{Q}^\times_p  \ar@{->>}[u] \ar@{->>}[d] & = \mathrm{PGL}_2(\mathbb{Q}_p)\\
\mathbb{P}^1(\mathbb{Q}_p) & &  B^\times_p / \left( \mathbb{Q}^\times_p \cdot R^\times_{N^+p^{\infty}, p} \right) \ar@{->>}[d] \ar@{-->}[ll]_-{\textrm{boundary}} & \txt{the set of certain consecutive sequences\\ of homothety classes of lattices\\(A ``natural" domain for Gross points)}\\
& &  B^\times_p / \left( \mathbb{Q}^\times_p \cdot R^\times_{N^+,p} \right) \ar@{->>}[d] & \txt{the set of homothety classes of lattices in $\mathbb{Q}^2_p$}  \\
& & R[1/p]^\times \backslash B^\times_p /  R^\times_{N^+,p} & \textrm{the domain of weight two quaternionic forms} .
}
\]
The na\"{i}ve analogy suggests us considering $B^\times \backslash \widehat{B}^\times  / \widehat{R}^{\times}_{N^{+}p^{\infty}} \simeq B^\times_p / \left( \mathbb{Q}^\times_p \cdot R^\times_{N^+p^{\infty}, p} \right)$ or even $\mathrm{PGL}_2( \mathbb{Q}_p )$ as the domain of the quaternionic forms, but \emph{it is not true}.
Thus, the na\"{i}ve analogy does not gives us a chance to find an analogous element of $( i \infty ) - (0) \in \mathrm{Div}^0(\mathbb{P}^1(\mathbb{Q}))$ in $B^\times \backslash \widehat{B}^\times  / \widehat{R}^{(p), \times}$ if the weight of the form $> 2$.
\begin{rem}[on the relation with the $p$-adic upper half plane]
The relation between the $p$-adic upper half plane $\mathfrak{h}_p$ and $B^\times_p / \left( \mathbb{Q}^\times_p \cdot R^\times_{N^+ p^\infty, p} \right)$ is well-documented in \cite[$\S$1]{dasgupta-teitelbaum}. We also suggest \cite{schneider} and \cite{teitelbaum} for the theory of boundary distributions on $\mathbb{P}^1(\mathbb{Q}_p)$. 
\end{rem}
As in the picture, the domain of quaternionic forms lies ``deeper" than $B^\times_p / \left( \mathbb{Q}^\times_p \cdot R^\times_{N^+p^\infty, p} \right)$ if their weight is $> 2$, and even the domain
$B^\times \backslash \widehat{B}^\times  / \widehat{R}^{(p), \times}$
has no direct geometric description as far as we know.
In the case of weight two forms, it suffices to find (classical) Gross points on $B^\times_p / \left( \mathbb{Q}^\times_p \cdot R^\times_{N^+,p} \right)$; thus, the na\"{i}ve analogy works well and the classical Gross points can be lifted to geometric Gross points on $B^\times_p / \left( \mathbb{Q}^\times_p \cdot R^\times_{N^+p^\infty, p} \right)$. 
Although it seems difficult to find a geometric motivation, Chida and Hsieh directly and explicitly constructed Gross points on $\widehat{B}^\times$ in \cite{chida-hsieh-p-adic-L-functions}. Their construction allows us to find the analogue of $(i\infty) - (0)$ for the quaternionic setting.
We review their explicit construction of Gross points (``explicit Gross points") in $\S$\ref{sec:explicit_Gross_points}, give them an geometric interpretation for the case of weight two forms (``geometric Gross points") in $\S$\ref{sec:geometric_gross_pts}, and compare these points in $\S$\ref{sec:comparison_gross_pts}.

\subsection{Control theorems}
In $\S$\ref{sec:control_proof}, we reprove the following control theorem for non-critical slope forms.
\begin{thm}[Theorem {\ref{thm:control}}] \label{thm:main_thm_1_classicality}
There exist Hecke-equivariant isomorphisms
\begin{align*}
S^{N^-} (N^+p, \mathscr{D}_k )^{(< k-1)} & \overset{\simeq}{\to} S^{N^-} (N^+p, \mathbf{D}_k )^{(< k-1)} \\
& \overset{\simeq}{\to} S^{N^-}_{k} (N^+p, E)^{(< k-1)} .
\end{align*}
\end{thm}
\begin{rem}
Theorem \ref{thm:main_thm_1_classicality} is a quaternionic analogue of \cite[Theorem 1.1 and Theorem 5.12]{pollack-stevens}) and \emph{generalizes} \cite[$\S$3]{longo-vigni-control} to the non-critical slope case.
\end{rem}

For the slope zero subspace, we obtain an integrally refined control theorem, which \emph{refines} \cite[$\S$3]{longo-vigni-control}.
\begin{cor}[Corollary \ref{cor:control_ordinary}]
There exist Hecke-equivariant isomorphisms
\begin{align*}
S^{N^-} (N^+p, \mathscr{D}_k(\mathcal{O}) )^{(0)} & \overset{\simeq}{\to} S^{N^-} (N^+p, \mathbf{D}_k(\mathcal{O}) )^{(0)} \\
& \overset{\simeq}{\to} S^{N^-}_{k} (N^+p, \mathcal{O})^{(0)} .
\end{align*}
\end{cor}

\subsection{Overconvergent construction of $p$-adic $L$-functions}
Using the explicit Gross points, we are able to construct anticyclotomic $p$-adic $L$-functions of modular forms of non-critical slope.
The following theorem generalizes the constructions of Bertolini-Darmon \cite{bertolini-darmon-mumford-tate-1996}, \cite{bertolini-darmon-imc-2005}, Bertolini-Darmon-Iovita-Spiess \cite{bertolini-darmon-iovita-spiess-2002}, and Chida-Hsieh \cite{chida-hsieh-p-adic-L-functions}. This also can be regarded as a quaternionic analogue of \cite[$\S$6]{pollack-stevens}.

\begin{thm} \label{thm:main_thm_2_overconvergent_construction}
Let $f_k$ be a newform of slope $h < k-1$ and $\Phi_{f_k}$ be the corresponding overconvergent quaternionic form. Then there exists an element $\varsigma^{(1)} \in \widehat{B}^\times$ (Definition \ref{def:explicit_gross_pts}) such that $\Phi_{f_k} ( \varsigma^{(1)} )$ is the $h$-admissible distribution (Definition \ref{def:the_distribution}) which defines an half of the anticyclotomic $p$-adic $L$-functions of $f_k$ (Definition \ref{defn:p-adic_L-functions}) and satisfies the expected interpolation property (Corollary \ref{cor:interpolation_p-adic_theta}).
\end{thm}
\subsection{Comparison with the former work}
We summarize the comparison with the former work.
\begin{itemize}
\item Gross proved the interpolation formula for weight two forms of prime level with the twist by unramified ring class character in \cite{gross}, and the formula is generalized to the weight two forms of arbitrary level and ring class characters of arbitrary conductor and finite order in \cite[Theorem 7.1]{zhang-gz2}.
\item In \cite{bertolini-darmon-mumford-tate-1996}, \cite{bertolini-darmon-imc-2005}, the anticyclotomic $p$-adic $L$-functions for $p$-ordinary $p$-stabilized newforms of weight two with the twist by ring class characters of $p$-power conductor and of finite order are constructed via a Stickelberger type argument.
\item In \cite{bertolini-darmon-iovita-spiess-2002}, the construction generalizes to $p$-newforms (the exceptional zero case) of even weight with the twist of unramified ring class characters of finite order. It can be regarded as an overconvergent construction due to \cite[(8)]{bertolini-darmon-iovita-spiess-2002} using the $p$-adic integration on $\mathbb{P}^1(\mathbb{Q}_p)$ \`{a} la Schneider-Teitelbaum. In this construction, a property of $p$-newforms is used essentially.
The interpolation formula for higher weight forms is given in \cite[Proposition 2.16]{bertolini-darmon-iovita-spiess-2002} only for unramified character twists, and the formula for ring class characters of $p$-power conductor is stated as a conjecture \cite[Conjecture 2.17]{bertolini-darmon-iovita-spiess-2002}. Indeed, \cite[Proposition 4.3]{chida-hsieh-p-adic-L-functions} proves \cite[Conjecture 2.17]{bertolini-darmon-iovita-spiess-2002} as stated in \cite[Remark after Proposition 4.3]{chida-hsieh-p-adic-L-functions}. See also \cite{yuan-thesis}.
\item  In \cite{bertolini-darmon-hida-2007}, the construction generalizes to $p$-ordinary $p$-stabilized newforms but it only allows genus characters \cite[page 412]{bertolini-darmon-hida-2007} for the character twist. The construction depends heavily on a quaternionic variant of Hida theory and the Hida theory there does not preserve the integrality.
\item In \cite{chida-hsieh-p-adic-L-functions}, their construction works for $p$-ordinary $p$-stabilized newforms with limitation of weight $k < p+2$ but with much more general twists by any locally algebraic $p$-adic characters of weight $(i, -i)$ with $-k/2 < i < k/2$ as described in \cite[Introduction]{chida-hsieh-p-adic-L-functions}. Here, the restriction of weight comes from the integrality and $\mu$-invariant issues. Also, Gross points are explicitly constructed at the level of $\widehat{B}^\times$. It is very important in our construction.
\item In \cite{castella-longo}, \cite{castella-kim-longo}, the construction works for $p$-ordinary $p$-stabilized newforms with the twist by same characters as in \cite{chida-hsieh-p-adic-L-functions}. This method uses an integral quaternionic Hida theory and big Gross points.
\item In this article, the construction works for $p$-stabilized newforms of non-critical slope and allows the twist by any locally algebraic $p$-adic character of weight $(i, -i)$ with $-k/2 < i < k/2$ arising from an anticyclotomic Hecke character. (cf.~\cite[$\S$4.2]{chida-hsieh-p-adic-L-functions}.)
 If the form is ordinary, then more character twists are allowed as in the case of \cite[Theorem 4.6]{chida-hsieh-p-adic-L-functions}. However, the interpolation formula is given only by $p$-power congruences (Corollary \ref{cor:interpolation_p-adic_theta}) unless the form is ordinary or of weight two.
\end{itemize}

\subsection{Organization}
In $\S$\ref{sec:explicit_Gross_points}, we review the explicit construction of Gross points following Chida-Hsieh, which is a key input of this work.
In $\S$\ref{sec:geometric_gross_pts}, we give a geometric interpretation of the explicit Gross points for the case of weight two forms.
In $\S$\ref{sec:comparison_gross_pts}, we compare these two Gross points. We also review other descriptions of Gross points.
In $\S$\ref{sec:coefficients}, we fix the convention of the coefficient modules for quaternionic forms.
In $\S$\ref{sec:automorphic_forms}, we review quaternionic forms, introduce their overconvergent variants, and prove the control theorem (Theorem \ref{thm:control}).
In $\S$\ref{sec:construction}, we give the overconvergent construction of the distribution (Definition \ref{def:the_distribution}) using the explicit Gross point, which is an half of the $p$-adic $L$-function. Also we recover classical theta elements from the distribution.
In $\S$\ref{sec:weak_interpolation}, we prove the ``weak" interpolation formula (Corollary \ref{cor:interpolation_p-adic_theta}) for the distribution using the formula of Chida-Hsieh (Theorem \ref{thm:interpolation_complex_theta}).
In $\S$\ref{sec:speculations}, we give some speculations and ask questions we do not have answers yet. 
%
%
%
%
%
\section{Explicit Gross points \`{a} la Chida and Hsieh} \label{sec:explicit_Gross_points}
We very closely follow \cite[$\S$2.1 and $\S$2.2]{chida-hsieh-p-adic-L-functions} for the explicit construction.
The novelty of this explicit construction of Gross points given by Chida and Hsieh is that the points lie \emph{at the level of $\widehat{B}^\times$}. This allows us to consider the Gross points at the ``deepest" level.
This explicit approach shows us that it seems more natural to look at the ``spaces at certain infinite levels" for the construction of $p$-adic $L$-functions.

Also, in the case of weight two forms, these Gross points can be realized purely geometrically in terms of the Bruhat-Tits tree for $\mathrm{PGL}_2(\mathbb{Q}_p)$. We will see this in the next section ($\S$\ref{sec:geometric_gross_pts}).

\begin{rem}[on the tame level structure on the domain]
The domain of modular symbols $\mathrm{Div}^0(\mathbb{P}^1 ( \mathbb{Q} )  )$ is completely independent of level structure and the information of the level structure entirely lies in congruence subgroups.
However, the domain of quaternionic forms $B^\times \backslash \widehat{B}^\times / \widehat{R}^{(p), \times}$ depends on its tame level structure obviously. Thus, the shape of Gross points depends on the tame level structure.
\end{rem}

\subsection{Explicit setup} \label{subsec:explicit_setup}
Let $K$ be the imaginary quadratic field of discriminant $-D_K <0$.
Define
$$\vartheta := 
\left \lbrace
    \begin{array}{ll}
     \dfrac{ D_K - \sqrt{-D_K} }{ 2 }  & \textrm{ if } 2 \nmid D_K \\ 
\dfrac{ D_K - 2\sqrt{-D_K} }{ 4 }  & \textrm{ if } 2 \mid D_K
    \end{array}
    \right.
$$
 so that
$\mathcal{O}_K = \mathbb{Z} + \mathbb{Z}\vartheta$.

Let $B$ be the definite quaternion algebra over $\mathbb{Q}$ of discriminant $N^-$ and level $N^+$ under Assumption \ref{assu:parity}.
Then there exists an embedding of $K$ into $B$ (\cite[$\S$3 of chapitre II and $\S$5.C of chapitre III]{vigneras}).
More explicitly, we choose a $K$-basis $(1,J)$ of $B$ so that $B = K \oplus K \cdot J$ such that
\begin{enumerate}
\item $\beta := J^2 \in \mathbb{Q}^\times$ with $\beta <0$,
\item $J \cdot t = \overline{t} \cdot J$ for all $t \in K$,
\item $\beta \in \left( \mathbb{Z}^\times_q \right)^2$ for all $q \mid pN^+$,
\item $\beta \in \mathbb{Z}^\times_q$ for all $q \mid D_K$.
\end{enumerate}
Fix a square root $\sqrt{\beta} \in \overline{\mathbb{Q}}$ of $\beta$.
Fix an isomorphism
$$i := \prod i_q : \widehat{B}^{(N^-)} \simeq \mathrm{M}_2(\mathbb{A}^{(N^-\infty)})$$
as follows:
\begin{enumerate}
\item For each finite place $q \mid N^+p$, the isomorphism
$$i_q : B_q \simeq \mathrm{M}_2(\mathbb{Q}_q)$$ by
\begin{align*}
i_q(\vartheta) & = \left( \begin{matrix}
\mathrm{trd}(\vartheta) & - \mathrm{nrd}(\vartheta) \\
1 & 0
\end{matrix} \right)  \\
i_q(J) & = \sqrt{\beta} \cdot \left( \begin{matrix}
-1 & \mathrm{trd}(\vartheta) \\
0 & 1
\end{matrix} \right) 
\end{align*}
where $\mathrm{trd}$ and $\mathrm{nrd}$ are the reduced trace and the reduced norm on $B$, respectively.
Note that $\sqrt{\beta} \in \mathbb{Z}^\times_q$ here.
\item For each finite place $q \nmid pN^+$, the isomorphism
$$i_q : B_q \simeq \mathrm{M}_2(\mathbb{Q}_q)$$ is chosen so that
$$i_q \left( \mathcal{O}_K \otimes \mathbb{Z}_q  \right) \subseteq \mathrm{M}_2(\mathbb{Z}_q) .$$
\end{enumerate}
We fix an embedding
$i_K : B \hookrightarrow \mathrm{M}_2(K)$ defined by $a+ bJ  \mapsto \left( \begin{matrix} a & b\beta \\ \overline{b} & \overline{a} \end{matrix}  \right)$ and define $i_\mathbb{C} : B \hookrightarrow \mathrm{M}_2(\mathbb{C})$ by $\iota_{\infty} \circ i_K$.

\subsection{The construction of the points}
Fix a decomposition $N^+ \mathcal{O}_K = \mathfrak{N}^+ \cdot \overline{\mathfrak{N}^+}$, which corresponds to the choice of the orientations of local Eichler orders at primes dividing $N^+$.
We define the local Gross point $\varsigma_q \in B^\times_q$ for any rational prime $q$.
\subsubsection{$q \nmid N^+p$}
Let $q \nmid N^+p$ be a prime.
Then
$$\varsigma_q := 1$$ 
in $B^\times_q$.
\subsubsection{$q \mid N^+$}
Let $q \mid N^+$ be a prime, and write $q = \mathfrak{q} \overline{\mathfrak{q}}$ in $\mathcal{O}_K$.
Then
$$\varsigma_q := \frac{1}{\sqrt{D_K}}\cdot \left( \begin{matrix}
\vartheta & \overline{\vartheta} \\
1 & 1
\end{matrix} \right) \in \mathrm{GL}_2(K_\mathfrak{q}) = \mathrm{GL}_2(\mathbb{Q}_q) .$$ 
\subsubsection{$q=p$} \label{subsubsec:p-part_of_explicit_Gross_points}
Suppose that $p = \mathfrak{p}\overline{\mathfrak{p}}$ splits in $K$. Then we put
$$\varsigma^{(n)}_p = \left( \begin{matrix}
\vartheta & -1 \\
1 & 0
\end{matrix} \right)
\left( \begin{matrix}
p^n & 0 \\
0 & 1
\end{matrix} \right)
 \in \mathrm{GL}_2(K_\mathfrak{p}) = \mathrm{GL}_2( \mathbb{Q}_{p} ).$$

Suppose that $p$ is inert in $K$. Then we put
$$\varsigma^{(n)}_p = \left( \begin{matrix}
0 & 1 \\
-1 & 0
\end{matrix} \right)
\left( \begin{matrix}
p^n & 0 \\
0 & 1
\end{matrix} \right)
 \in \mathrm{GL}_2(K_p) = \mathrm{GL}_2( \mathbb{Q}_{p^2} ).$$

\subsubsection{Putting it all together} \label{subsubsec:adelic_explicit_Gross_points}
\begin{defn}[Explicit Gross points] \label{def:explicit_gross_pts}
We define \textbf{the explicit Gross point $\varsigma^{(n)}$ of conductor $p^n$ on $\widehat{B}^\times$} by
$$\varsigma^{(n)} := \varsigma^{(n)}_p \times \prod_{q \neq p} \varsigma_q \in \widehat{B}^\times .$$
\end{defn}

\subsection{Anticyclotomic Galois action on Gross points} \label{subsec:anticyclotomic_action}
We define the map
$$x_n : \widehat{K}^\times \to \widehat{B}^\times$$
by
$$x_n( \xi) := \xi \cdot \varsigma^{(n)} $$
where the action is given by the embedding $K$ into $B$ chosen in $\S$\ref{subsec:explicit_setup}.
Then the set of all points $\left\lbrace x_n( \xi ) : \xi \in \mathbb{A}^{(\infty), \times}_K \right\rbrace$ is called \textbf{twisted explicit Gross points of conductor $p^n$}.

We also define the map
$$x^{-1}_n : \widehat{K}^\times \to \widehat{B}^\times$$ by
$$x^{-1}_n( \xi) := \xi^{-1} \cdot \varsigma^{(n)} .$$
Note that $x^{-1}_n(\xi)$ does not appear as $x_n(\xi)$ since the inverse map $K^\times \to K^\times$ by $x \mapsto 1/x$ does not propagate its domain to $K$.
\subsection{Families of optimal embeddings} \label{subsec:families_of_optimal_embeddings}
Let $\mathcal{O}_{n} = \mathbb{Z} + p^n \mathcal{O}_K$ be the order of $K$ of conductor $p^n$.
Let $R_{N^+p^r}$ be an Eichler order of level $N^+p^r$ prime to $N^-$.
By the argument in \cite[$\S$2.2]{chida-hsieh-p-adic-L-functions}, the embedding of $K$ into $B$ is an optimal embedding of $\mathcal{O}_{n}$ into the Eichler order $B \cap \varsigma^{(n)} \cdot \widehat{R}^\times_{N^+p^r} \cdot \left( \varsigma^{(n)} \right)^{-1}$ if $r \leq n$, i.e.
$$\left( B \cap \varsigma^{(n)} \cdot \widehat{R}^\times_{N^+p^r} \cdot \left( \varsigma^{(n)} \right)^{-1} \right) \cap K = \mathcal{O}_{n} .$$
This is used in the comparison among Gross points defined on different domains ($\S$\ref{subsec:other_gross_pts}).
\begin{rem}
See \cite[$\S$4.1.(12)]{longo-vigni-manuscripta} for another recipe of the families of optimal embeddings. Their recipe calculates the $p$-part only, but the oriented optimal embeddings are determined locally (\cite[Lemma 4.1]{longo-vigni-manuscripta}).
\end{rem}
\section{Geometric Gross points for weight two forms} \label{sec:geometric_gross_pts}
In this section, we give a geometric interpretation of the projection of the explicit Gross points to the double coset space at the $\Gamma_0(p^\infty)$-level, i.e.~the Bruhat-Tits tree. Although it does not gives us the full information of the explicit Gross points, it seems to be helpful for a more theoretical understanding of the Gross points. The geometric description naturally shows us a $p$-adic intuition in the construction of anticyclotomic $p$-adic $L$-functions of modular forms (of weight two, at least).

In the construction of the anticyclotomic $p$-adic $L$-functions of an ordinary newform of weight two, we choose an infinite sequence of consecutive vertices $v_0, v_1, \cdots, v_n, \cdots$ without backtracking on $BT_p$
(We call them \textbf{Gross points of conductor $p^n$ at level 0}).
The oriented edge $e_n$ on $BT_p$ whose source is $v_{n}$ and target is $v_{n+1}$ is called \textbf{Gross points of conductor $p^n$ at level 1}. By construction, the sequence of the edges $e_0, e_1, \cdots, e_n, \cdots$ has coherent direction. We also call these two points \textbf{classical Gross points}.
\begin{rem}
Indeed, the first choice $v_0$ is an infinite line if $p$ splits in $K$, but we will call it a ``vertex" for convenience.
 See \cite[Figure 1]{darmon-iovita} for the picture. 
\end{rem}
The goal of this section is to reinterpret these infinite choices of classical Gross points at the $\Gamma_0(p^r)$ and $\Gamma_0(p^\infty)$-levels. The ``Gross points at the $\Gamma_0(p^\infty)$-level" will be called \textbf{geometric Gross points} (Definition \ref{def:geometric_gross_pts}).

\subsection{Galois-theoretic setup}
Let $\mathcal{O}_K$ be the ring of integers of $K$ and $\mathcal{O}_K[\frac{1}{p}]$ be the maximal $\mathbb{Z}[\frac{1}{p}]$-order in $K$. Let 
$$\widetilde{G}_\infty  = K^\times \backslash \widehat{K}^\times / (\widehat{\mathbb{Q}}^\times \cdot \prod_{\ell\neq p} \mathcal{O}_K[\frac{1}{p}]^\times_\ell )$$
be the Galois group of the ring class field $\widetilde{K}_\infty$ of $K$ of conductor $p^\infty$ so that $\widetilde{G}_\infty = \mathrm{Gal}(\widetilde{K}_\infty/K)$.
\begin{choice} \label{choice:embedding}
We choose an oriented optimal embedding
$\Psi_0 : K \to B$ such that
$\Psi_0 ( K ) \cap R[1/p] = \Psi_0 ( \mathcal{O}_K[1/p] )$
which is as equivalent as the choice in $\S$\ref{subsec:explicit_setup}.
\end{choice}
Then $\Psi_0$ induces a family of optimal embeddings
$\Psi_{p^r}$ such that
$\Psi_{p^r} ( K ) \cap R_{N^+p^r} = \Psi_{p^r} ( \mathcal{O}_{r} )$ for all $r \geq 0$ as in $\S$\ref{subsec:families_of_optimal_embeddings}.

\subsection{Classical Gross points: Gross points at level 0 and 1} \label{subsec:classical_gross_pts}
Let $BT_p$ be the Bruhat-Tits tree for $\mathrm{PGL}_2(\mathbb{Q}_p)$. The action of $\mathrm{PGL}_2(\mathbb{Q}_p)$ on $BT_p$ is given via right conjugation.
The chosen embedding $\Psi_0$ induces the $p$-adic embedding
$(\Psi_0)_p : K^\times_p / \mathbb{Q}^\times_p \hookrightarrow B^\times_p / \mathbb{Q}^\times_p \simeq \mathrm{PGL}_2(\mathbb{Q}_p)$.
This embedding yields the action of $K^\times_p / \mathbb{Q}^\times_p$ on $BT_p$ via left translation.
The structure of $\widetilde{G}_\infty$ can be described by the following short exact sequence:
\[
\xymatrix{
  & K^\times_p / \mathbb{Q}^\times_p  \ar[d] \ar[dr] \\
1 \ar[r] & G_\infty := K^\times_p / \mathbb{Q}^\times_p (\mathcal{O}_K[\frac{1}{p}])^\times \ar[r] & \widetilde{G}_\infty \ar[r] & \mathrm{Cl}(\mathcal{O}_K[\frac{1}{p}]) \ar[r] & 1
 }
\]

\begin{rem}
The class group $\mathrm{Cl}(\mathcal{O}_K[\frac{1}{p}])$ permutes oriented optimal embeddings transitively, and the permutation is explicitly defined in \cite[$\S$2.3]{bertolini-darmon-iovita-spiess-2002}.
\end{rem}

Consider the decreasing filtration of $K^\times_p / \mathbb{Q}^\times_p$
$$\cdots \subseteq U_{n+1} \subseteq U_n \subseteq U_{n-1} \subseteq \cdots  \subseteq U_1 \subseteq U_0  \subseteq K^{\times}_p/ \mathbb{Q}^{\times}_p$$
where $U_0$ is the maximal compact subgroup of $K^\times_p / \mathbb{Q}^\times_p$ and $U_n  =  (1+p^n \mathcal{O}_K \otimes_\mathbb{Z} \mathbb{Z}_p) / (1+p^n\mathbb{Z}_p)$ for each $n \geq 1$. Let $G_n := G_\infty /$(the image of $U_n$ in $G_\infty$) and $\widetilde{G}_n := \widetilde{G}_\infty /$(the image of $U_n$ in $\widetilde{G}_\infty$).

\begin{choice} \label{choice:sequence}
We choose a sequence of consecutive vertices $v_0, v_1, v_2, \cdots$ of $V(BT_p)$ with the coherent orientation and without backtracking such that
$\mathrm{Stab}_{(\Psi_0)_p(K^\times_p / \mathbb{Q}^\times_p)} (v_n) = U_n$
for all $n \geq 0$
as equivalent as the choice in $\S$\ref{subsubsec:p-part_of_explicit_Gross_points} and $\S$\ref{subsubsec:adelic_explicit_Gross_points}.
\end{choice}

\begin{defn}[Classical Gross points]  \label{def:classical_Gross_pts} $ $
\begin{enumerate}
\item Each vertex $v_n$ in the chosen sequence is called \textbf{a (classical) Gross point of conductor $p^n$ at level 0}.
\item Each oriented edge $e_n = (v_n \to v_{n+1})$ is called \textbf{a (classical) Gross point of conductor $p^n$ at level 1}.
\end{enumerate}
\end{defn}

\begin{rem}[on the domain of Gross points]
In many literature, classical Gross points are defined on the quotient graphs of $BT_p$, which are equivalent to the double coset spaces via strong approximation. However, both Gross points on the tree and on the quotient graph give the exactly same result since quaternionic forms are invariant under the quotient.
It seems difficult to observe Gross points at higher level on the quotient graph intuitively since the images of length $n$ line segments in the quotient graph may have very random shapes due to the complication of the quotient graph. For the computation of the graph, see \cite{computing-graphs}.
Recently, it seems that this complication has an application to cryptography, so called isogeny based cryptography.
For example, see \cite[$\S$2.2]{isogeny_crypto}.
\end{rem}

\subsection{A simple observation: towards higher and infinite level} \label{subsec:a_simple_observation}
A natural idea toward the Gross points on a certain space at the infinite level begins with the following question.
\begin{ques}
How can we regard a coherent infinite sequence of classical Gross points itself as one element in a more suitable domain than the set of vertices or oriented edges of the Bruhat-Tits tree?
\end{ques}

We recall a strong approximation result.
\begin{prop}[{\cite[$\S$1.2.(16)]{bertolini-darmon-imc-2005}}]  \label{prop:strong_approx}
The embedding into the $p$-th place
\begin{align*} 
(R[\frac{1}{p}])^\times \backslash B^{\times}_p / \mathbb{Q}^{\times}_p & \simeq  B^{\times} \backslash \widehat{B}^{\times} / \left( \widehat{\mathbb{Q}}^{\times} \prod_{\ell \neq p} R^\times_\ell \right) \\
b_p & \mapsto  (1, \cdots, 1, b_p, 1, \cdots )
\end{align*}
is a canonical bijection.
\end{prop}

Let $\overset{\to}{E}_r (BT_p)$ be the set of consecutive line segments of length $r$ with coherent orientation of $BT_p$ without backtracking.
Let $v_0, v_1, v_2, \cdots, v_r$ be the sequence of consecutive vertices of $BT_p$ whose stabilizers are $\mathbb{Q}^\times_p \cdot \mathrm{GL}_2(\mathbb{Z}_p)$, $\mathbb{Q}^\times_p \cdot \gamma^{-1} \mathrm{GL}_2(\mathbb{Z}_p) \gamma$, $\cdots$, $\mathbb{Q}^\times_p \cdot \gamma^{-r} \mathrm{GL}_2(\mathbb{Z}_p) \gamma^{r}$
with $\gamma = \left( \begin{smallmatrix}
p & 0 \\ 0 & 1
\end{smallmatrix} \right)$, respectively. Thus, the whole sequence is an element of $\overset{\to}{E}_r (BT_p)$.
Then the stabilizer of the whole sequence is $\mathbb{Q}^\times_p \cdot R^\times_{r,p}$. We observe the following statement.
\begin{prop} \label{prop:orbit}
The $\mathrm{GL}_2(\mathbb{Q}_p)$-orbit of the sequence of consecutive vertices $v_0, v_1, v_2, \cdots, v_r$ without backtracking is $\overset{\to}{E}_r (BT_p)$.
\end{prop}
\begin{proof}
Since the action of $\mathrm{GL}_2(\mathbb{Q}_p)$ on $BT_p$ preserves distance, the $\mathrm{GL}_2(\mathbb{Q}_p)$-orbit of the sequence $v_0, v_1, v_2, \cdots, v_r$ is a subset of $\overset{\to}{E}_r (BT_p)$.
It suffices to show the action of $\mathrm{GL}_2(\mathbb{Q}_p)$ on $\overset{\to}{E}_r (BT_p)$ is transitive.
Let $w_0, w_1, w_2, \cdots, w_{r-1}, w_r$ be an arbitrary element of $\overset{\to}{E}_r (BT_p)$. 
Since $\mathrm{GL}_2(\mathbb{Q}_p)$ acts transitively on $\overset{\to}{E}_1 (BT_p)$,
we may assume $w_{r-1}  = v_{r-1}$ and $w_r = v_r$.
Now we apply induction on $r$ in the decreasing direction.
Let $k < r$ be the smallest integer such that $w_k = v_k, w_{k+1} = v_{k+1}, \cdots,$ and $w_r = v_r$.
The stabilizer of the sequence $w_k, \cdots w_r$ consists of the matrices
$\left( \begin{smallmatrix}
a & b/p^k \\ p^r c & d
\end{smallmatrix} \right)
\in \mathrm{GL}_2(\mathbb{Q}_p)$ with $a, b, c, d \in \mathbb{Z}_p$.

Then $w_{k-1}$ corresponds to the homothety class of the lattice
$$\left( \mathbb{Z}_p \times \mathbb{Z}_p \right) \cdot \left( \begin{smallmatrix}
p^k & \Delta \\ 0 & p
\end{smallmatrix} \right)$$
for some $\Delta = 1, \cdots p-1$.
Multiplying $\left( \begin{smallmatrix}
1 & -\Delta/p^k \\ 0 & 1
\end{smallmatrix} \right)$ on the right, the lattice corresponding to $w_{k-1}$ changes to
the lattice 
$\left( \mathbb{Z}_p \times \mathbb{Z}_p \right) \cdot \left( \begin{smallmatrix}
p^{k-1} & 0 \\ 0 & 1
\end{smallmatrix} \right)$ upto homothety.
Since the corresponding vertex is $v_{k-1}$ and $\left( \begin{smallmatrix}
1 & -\Delta/p^k \\ 0 & 1
\end{smallmatrix} \right)$ is in the stabilizer of the sequence $v_k, \cdots, v_r$, we reduce $k$ to $k-1$ by multiplying $\left( \begin{smallmatrix}
1 & -\Delta/p^k \\ 0 & 1
\end{smallmatrix} \right)$ on the right. Repeating the process, we obtain the conclusion.
\end{proof}
Then Proposition \ref{prop:orbit} and the orbit-stabilizer theorem show that there exist bijections:
\[
\xymatrix@C=8em{
\overset{\to}{E}_r (BT_p) \ar[r]^-{\textrm{Prop.~\ref{prop:orbit}}}_-{\simeq} & (v_0, \cdots, v_r) \cdot \mathrm{GL}_2(\mathbb{Q}_p) \ar[r]^-{\textrm{orbit-stabilizer}}_-{\simeq} & \mathrm{GL}_2(\mathbb{Q}_p) / \left( \mathbb{Q}^\times_p \cdot  R^{\times}_{r,p}  \right).
}
\]
This identification gives us a hint to define the case of $r =\infty$.

We define
$$\overset{\to}{E}_\infty (BT_p) := \mathrm{GL}_2(\mathbb{Q}_p) / \left( \mathbb{Q}^\times_p \cdot  R^\times_{N^+p^\infty,p}  \right) \simeq (v_0, v_1, \cdots) \cdot \mathrm{GL}_2(\mathbb{Q}_p).$$
Then each element here has interpretation as an infinite consecutive sequence of vertices from a vertex to a boundary of the Bruhat-Tits tree since each element has the form $(v_0, v_1, \cdots) \cdot \gamma$ where $\gamma \in \mathrm{GL}_2(\mathbb{Q}_p)$.
Also, $\overset{\to}{E}_\infty (BT_p)$ admits natural quotient maps
\[
\xymatrix@R=0em{
\overset{\to}{E}_\infty (BT_p) \ar[r] &  \overset{\to}{E}_r (BT_p) \\
(v_0, v_1, \cdots) \ar@{|->}[r] &  (v_0, v_1, \cdots, v_r)
}
\]
for all $r \geq 0$.
Note that the stabilizer of $v_0$ is $U_0$ and the stabilizer of a boundary of $BT_p$, an element in $\mathbb{P}^1(\mathbb{Q}_p)$, is trivial.

We are now able to give a heuristic definition of geometric Gross points and a more group-theoretic and axiomatic definition is given in $\S$\ref{subsec:group-theoretic-defn}.
\begin{defn}[Heuristic definition of geometric Gross points] \label{def:heuristic}
A \textbf{geometric Gross point} is the consecutive sequence of classical Gross points at level 0 depending on Choice \ref{choice:embedding}, Choice \ref{choice:sequence}, and Definition \ref{def:classical_Gross_pts} in $\overset{\to}{E}_\infty (BT_p)$.
\end{defn}

\subsection{A group theoretic realization and independence} \label{subsec:group-theoretic-defn}
We give a more axiomatic definition of geometric Gross points.
For notational convenience, let
\begin{itemize}
\item $G = \mathrm{GL}_2(\mathbb{Q}_p)$,
\item $B = $ the upper Borel subgroup of $G$,
\item $K_0 = \mathrm{GL}_2(\mathbb{Z}_p)$,
\item $K_r = \Gamma(p^r\mathbb{Z}_p)$, and
\item $Z =$ the center of $G \simeq \mathbb{Q}^\times_p$.
\end{itemize}
Then the Iwasawa decomposition implies that $G = B \cdot K$.
We have natural projection maps
\[
\xymatrix{
G / \left( B \cap K_0Z \right) \ar[r] \ar[d] & G/B \ar[r]^-{\simeq}& \mathbb{P}^1(\mathbb{Q}_p)\\
G/K_0 Z \ar[d]^-{\simeq}\\
V(BT_p)
}
\]
and embedding
\[G / \left( B \cap K_0Z \right) \hookrightarrow V(BT_p) \times \mathbb{P}^1(\mathbb{Q}_p) .\]
We remark that we do not know how to characterize the image explicitly.
 
We axiomatize Definition \ref{def:heuristic}.
\begin{defn}[Geometric Gross points] \label{def:geometric_gross_pts}
Fix a central Gross point/line $v_0$ of conductor 1 and level 0.
For $n \geq 0$, an element $\sigma^{(n)} \in G / \left( B \cap K_0 Z \right)$ is a  \textbf{geometric Gross point of conductor $p^n$} if
\begin{enumerate}
\item the image of $\sigma^{(0)}$ in $G/K_0 Z$ is $v_0$ under the natural projection,
\item the image of $\sigma^{(n)} := \sigma^{(0)} \cdot \left( \begin{matrix} p^n & 0 \\ 0 & 1 \end{matrix}\right)$ in $G/K_0 Z$ has stabilizer $U_n$ under the action of $K^\times_p / \mathbb{Q}^\times_p$ via the chosen optimal embedding $\Psi_0$, and
\item the image of $\sigma^{(n)}$ in $G/B$ does not change for all $n \geq 0$.  
\end{enumerate}
\end{defn}
\begin{prop}[Uniqueness and Galois properties] $ $ \label{prop:properties_of_gross_pts}
\begin{enumerate}
\item By the embedding, Property (1) and (3) in Definition \ref{def:geometric_gross_pts} uniquely determines a point in $G / \left( B \cap K_0 Z \right)$.
\item The image of $\sigma^{(r)}$ in $G/ (Z \cdot R^\times_{N^+p^r,p})$ is $(v_0, \cdots, v_r)$ where $v_i = v_{i-1} \cdot \left( \begin{matrix} p & 0 \\ 0 & 1 \end{matrix}\right)$.
\end{enumerate}
\end{prop}
\begin{proof} $ $
\begin{enumerate}
\item Obvious since $G / \left( B \cap K_0Z \right) \hookrightarrow V(BT_p) \times \mathbb{P}^1(\mathbb{Q}_p) $.
\item The image of $\sigma^{(0)}$ under the natural quotient map
$G / \left( B \cap K_0Z \right) \to G/ K_0Z$ is $v_0$ and a lifing $v_0$ to $G/ (Z \cdot R^\times_{N^+p^r,p})$ gives a length $r$ line segment whose target endpoint is $v_0$.
Shifting by $\left( \begin{matrix} p^r & 0 \\ 0 & 1 \end{matrix}\right)$ gives the conclusion.
\end{enumerate}
\end{proof}
For a given optimal embedding $\Psi_0 : K^\times  \to B^\times$,
we define the reversed embedding $\Psi^{-1}_0 : K^\times  \to B^\times$ by $a  \mapsto \Psi_0(a^{-1})$.
Then, for $n \geq 0$, we define the \textbf{dual geometric Gross point to $\sigma^{(n)}$} by an element $\sigma^{(n),*}  \in G / \left( B \cap K_0 Z \right)$ as exactly same as Definition \ref{def:geometric_gross_pts} but with the reversed embedding $\Psi^{-1}_0$ in Condition (2).

\begin{prop}[Independence of choices]
Let $\sigma^{(0)}_1$, $\sigma^{(0)}_2$ be two geometric Gross points.
Then they differ only by the translation by an element of $K^\times_p / \mathbb{Q}^\times_p$.
\end{prop}
\begin{proof} We split the proof into two parts depending on whether $p$ is inert in $K$ or splits in $K$. This is a slightly refined version of \cite[Lemma 4.3]{bertolini-darmon-survey-2001}.
\begin{enumerate}
\item (The inert case)
From \cite[Lemma 2.7]{bertolini-darmon-cerednik-drinfeld-1998}, we can deduce $\widetilde{G}_n$ acts transitively on the classical Gross points of conductor $p^n$ for any $n$. With \cite[Lemma 2.8]{bertolini-darmon-cerednik-drinfeld-1998}, it is easy to see that $\widetilde{G}_n$ acts transitively on higher Gross points of conductor $p^n$.
Note that the subquotient $K^\times_p / \mathbb{Q}^\times_p$ of $\widetilde{G}_\infty$ acts on $\mathbb{P}^1(\mathbb{Q}_p)$ simply transitively due to \cite[$\S$4.1.Step 2]{bertolini-darmon-survey-2001}. This ensures that $\widetilde{G}_\infty$ acts on the set of Gross points at infinite level transitively. 
\item (The split case)
We can deduce the same conclusion for higher Gross points of conductor $p^n$ following the argument in \cite[$\S$3]{bertolini-darmon-uniformization-1999}. However, $K^\times_p / \mathbb{Q}^\times_p$ does \emph{not} act on $\mathbb{P}^1(\mathbb{Q}_p)$ transitively in this case. It has 3 orbits : 0, $\infty$, and $\mathbb{Q}^\times_p$. See \cite[$\S$7.I]{bertolini-darmon-uniformization-1999} for detail. However, any sequence of classical Gross points does not converges to 0 or $\infty$ in $\mathbb{P}^1(\mathbb{Q}_p)$ in this case. See \cite[$\S$2.2 and Figure 1]{darmon-iovita} for detail.
\end{enumerate}
\end{proof}

\section{Comparison among Gross points} \label{sec:comparison_gross_pts}
\subsection{Comparison of explicit and geometric Gross points}
Considering the strong approximation for quaternion algebras (Proposition \ref{prop:strong_approx}), we observe more precise relations among the double coset spaces as follows:
{\small
\[
\xymatrix{
& & \widehat{B}^\times \ar@{->>}[d] & \\
& &  B^\times \backslash \widehat{B}^\times  / \widehat{R}^{(p), \times}  \ar@{->>}[d] &  \txt{no level structure at $p$}\\
B^\times_p / \mathbb{Q}^\times_p \ar@{->>}[r] \ar@{->>}[d] & R[1/p]^\times \backslash B^{\times}_p / \mathbb{Q}^\times_p \ar[r]^-{\simeq}_-{\textrm{Prop.~\ref{prop:strong_approx}}} \ar@{->>}[d] &  B^\times \backslash \widehat{B}^\times  / \left( \mathbb{Q}^\times_p \cdot \widehat{R}^{(p), \times}  \right) \ar@{->>}[d]\\
B^\times_p / \left( \mathbb{Q}^\times_p \cdot R^\times_{N^+p^\infty, p} \right)  \ar@{->>}[d] \ar@{->>}[r] & R[1/p]^\times \backslash  B^\times_p / \left( \mathbb{Q}^\times_p \cdot R^\times_{N^+p^\infty, p} \right) \ar@{->>}[d] \ar[r]^-{\simeq}_-{\textrm{Prop.~\ref{prop:strong_approx}}} & B^\times \backslash \widehat{B}^\times  / \widehat{R}^{\times}_{N^+p^\infty} \ar@{->>}[d] &\txt{$\Gamma_0(p^\infty)$-level structure}\\
B^\times_p / \left( \mathbb{Q}^\times_p \cdot R^\times_{p} \right) \ar@{->>}[r] & R[1/p]^\times \backslash B^\times_p /  R^\times_{p} \ar[r]^-{\simeq}_-{\textrm{Prop.~\ref{prop:strong_approx}}} & B^\times \backslash \widehat{B}^\times  / \widehat{R}^{\times} & \txt{full level structure at $p$}
}
\]
}
Note that $B^\times \backslash \widehat{B}^\times  / \widehat{R}^{(p), \times}$ is the domain for quaternionic forms of arbitrary weight and $B^\times \backslash \widehat{B}^\times  / \widehat{R}^{\times}$ is the domain for quaternionic forms of weight two only.

Let
\begin{itemize}
\item $\varsigma^{(n)}$ be the explicit Gross point on $\widehat{B}^\times$ defined in Definition \ref{def:explicit_gross_pts},
\item $\sigma^{(n)}$ be the geometric Gross point on $B^\times_p / \left( \mathbb{Q}^\times_p \cdot R^\times_{N^+p^\infty, p} \right)$ defined in Definition \ref{def:geometric_gross_pts}, and
\item $v_n$ be the classical Gross point on $B^\times_p / \left( \mathbb{Q}^\times_p \cdot R^\times_{p} \right)$ defined in Definition \ref{def:classical_Gross_pts}.
\end{itemize}
The classical Gross points $v_n$ and the geometric Gross points coincide $\sigma^{(n)}$ by Proposition \ref{prop:properties_of_gross_pts}.(2).
The classical Gross points $v_n$ and the explicit Gross points $\varsigma^{(n)}$ coincide by the construction of the explicit Gross points and theta elements in \cite[$\S$4.1]{chida-hsieh-p-adic-L-functions}. (cf.~Choice \ref{choice:embedding}.)
Thus, these points coincide in the above diagram as follows.
\[
\xymatrix{
& & \varsigma^{(n)} \ar@{|->}[d] & \\
& & \varsigma^{(n)} \ar@{|->}[d] & \\
 & \varsigma^{(n)} \ar@{|->}[r] \ar@{|->}[d] & \varsigma^{(n)}  \ar@{|->}[d] \\
\sigma^{(n)} \ar@{|->}[d] \ar@{|->}[r] & \varsigma^{(n)} = \sigma^{(n)}  \ar@{|->}[d] \ar@{|->}[r] & \varsigma^{(n)} = \sigma^{(n)} \ar@{|->}[d]\\
\sigma^{(n)} = v_n \ar@{|->}[r] & \varsigma^{(n)} = \sigma^{(n)} = v_n \ar@{|->}[r] & \varsigma^{(n)} = \sigma^{(n)} = v_n
}
\]
\subsection{Comparison with other Gross points} \label{subsec:other_gross_pts}
We also consider other definitions of Gross points and the relation with them. All the Gross points here correspond to the classical one (of level 0).
\begin{defn}[Other definitions of Gross points] $ $
\begin{enumerate}
\item 
Consider the $K$-points of the Gross curve of level $N^+$ and discriminant $N^-$
 $$B^\times \backslash \widehat{B}^\times \times \mathrm{Hom}(K, B)/ \widehat{R}^{\times}_{N^+}.$$
Following \cite[$\S$2.1]{bertolini-darmon-mumford-tate-1996}, \cite[$\S$3.1]{longo-hilbert-modular-case}, 
 $(x_n, \Psi)$ is a \textbf{Gross point of conductor $p^n$ on the Gross curve} if 
$\Psi(K) \cap x_n\widehat{R}_{N^+}x^{-1}_n = \Psi(\mathcal{O}_n)$.
\item
Following \cite[$\S$5.3]{cornut-vatsal}, \cite[$\S$4.2]{longo-hilbert-modular-case}, we define the set of Gross points by
 $$\Psi_0(K^\times) \backslash \widehat{B}^\times / \widehat{R}^{\times}_{N^+}$$ 
and a  \textbf{Gross point $x_n \in \Psi_0(K^\times) \backslash \widehat{B}^\times / \widehat{R}^{\times}_{N^+}$ has conductor $p^n$} if
$\Psi_0(K) \cap x_n\widehat{R}_{N^+}x^{-1}_n = \Psi_0(\mathcal{O}_n)$.
\item The equivalence of the above descriptions comes from the map
$$\Psi_0(K^\times) \backslash \widehat{B}^\times / \widehat{R}^{\times}_{N^+} \to B^\times \backslash \widehat{B}^\times \times \mathrm{Hom}(K, B)/ \widehat{R}^{\times}_{N^+}$$
 defined by $x_n \mapsto (x_n, \Psi_0)$. See \cite[$\S$3.1]{longo-hilbert-modular-case} for proof.
\end{enumerate}
\end{defn}
\begin{rem}
Since we start with a chosen \emph{oriented} optimal embedding $\Psi_0$, all the ``CM points" in the original reference become Gross points.
\end{rem}
Then it is not difficult to check these Gross points coincide with the classical points (at least at the level of the values of quaternionic forms) by comparing two equivalent construction of theta elements of modular forms of weight two (\cite[$\S$1.2]{bertolini-darmon-imc-2005} with classical Gross points and \cite[$\S$2.7]{bertolini-darmon-mumford-tate-1996} with Gross points on Gross curves). Their relation can be summarized in the following diagram.
\[
\xymatrix{
\Psi_0(K^\times) \backslash \widehat{B}^\times / \widehat{R}^{\times}_{N^+} \ar[r] \ar[d] & B^\times \backslash \widehat{B}^\times / \widehat{R}^{\times}_{N^+}  & x_n \ar@{|->}[r] \ar@{|->}[d] & v_n\\
B^\times \backslash \widehat{B}^\times \times \mathrm{Hom}(K, B)/ \widehat{R}^{\times}_{N^+} & & (x_n, \Psi_0)
}
\]
\section{Coefficients} \label{sec:coefficients}
Let $E$ be a finite extension of $\mathbb{Q}_p$ and $\mathcal{O} = \mathcal{O}_E$.
Let
$$\Sigma_0(p) := \left\lbrace \left( \begin{smallmatrix} a & b \\ c & d \end{smallmatrix} \right) \in \mathrm{M}_2(\mathbb{Z}_p) : c \in p\mathbb{Z}_p, d \in \mathbb{Z}^\times_p,  \textrm{and} \ ad-bc \neq 0 \right\rbrace$$
be the semigroup we concern to see $U_p$-action. It is not as the same one as given in \cite[$\S$3.3]{pollack-stevens}.
More precisely, if
$r = \left( \begin{smallmatrix} a & b \\ c & d \end{smallmatrix} \right) \in \Sigma_0(p)$
then its adjugate
$r^* = \left( \begin{smallmatrix} d & -b \\ -c & a \end{smallmatrix}  \right)$ satisfies the condition given in \cite[$\S$3.3]{pollack-stevens}.
Since we will define our left action by the adjugate right action given in \cite[$\S$3.3]{pollack-stevens}, the specialization maps will be compatible with the convention of \cite{pollack-stevens}.
\subsection{Symmetric powers}
Our convention is similar to but not exactly same as those of \cite{chida-hsieh-p-adic-L-functions} and \cite{pollack-stevens}.
We also introduce an equivariant pairing to obtain the distribution relation later.
\subsubsection{Semigroup action}
Let
$L_k(E) := \mathrm{Sym}^{k-2}(E^2)$ and
$L_k(\mathcal{O}) := \mathrm{Sym}^{k-2}(\mathcal{O}^2)$.
They admit the \textbf{left} actions of $\Sigma_0(p)$ and $\mathrm{GL}_2(\mathbb{Z}_p)$
via the representation
$$\rho_k :  \Sigma_0(p) \ (\textrm{or } \mathrm{GL}_2(\mathbb{Z}_p)) \subseteq \mathrm{M}_2(\mathcal{O})  \to \mathrm{End}_{\mathcal{O}} (L_k (\mathcal{O}))$$
defined by
$$( \rho_k (r) \circ P ) (X,Y)  = P ( r^* (X,Y) ) = P(dX - bY, -cX + aY)$$
where $r = \left( \begin{smallmatrix} a & b \\ c & d \end{smallmatrix} \right) \in \Sigma_0(p)$ and $P(X,Y) \in L_k(\mathcal{O})$ is a homogeneous polynomial of variables $X, Y$ of degree $k-2$.

\subsubsection{An (ad hoc) equivariant pairing} \label{subsubsec:pairing}
Consider the following perfect $\mathrm{GL}_2(\mathbb{Q}_p)$-equivariant pairing:
\[
\xymatrix@R=0em{
\langle -, - \rangle_k : L_k \times L_k(\mathrm{det}^{2-k}) \ar[r] & E \\
 (X^i \cdot Y^{k-2-i}, X^{k-2-j} \cdot Y^{j}) \ar@{|->}[r] &  (- 1)^{i} \cdot \left( {\begin{matrix}
 k-2 \\ i
 \end{matrix}} \right)^{-1} \cdot \delta_{i,j}
}
\]
where $\delta_{i,j}$ is the Kronecker delta. The equivariant property is given as follows:
$$\langle \rho_k(r) \circ P_1(X,Y),\rho^*_k(r) \circ P_2(X,Y) \rangle_k = \langle  P_1(X,Y),  P_2(X,Y) \rangle_k $$
where $\rho^*_k := \rho_k \otimes \mathrm{det}^{2-k}$. We also write $L_k(2-k) = L_k(\mathrm{det}^{2-k})$.
\begin{rem}[Normalization of $\mathrm{GL}_2(\mathbb{Q}_p)$-action]
 In \cite[$\S$2.3]{chida-hsieh-p-adic-L-functions}, the action of $\mathrm{GL}_2(\mathbb{Q}_p)$ is unitarily normalized, i.e.~the action on the both side is given by $\rho_k \otimes \mathrm{det}^{\frac{2-k}{2}}$. However, the unitary normalization is not compatible with the integral theory of quaternionic forms. See $\S$\ref{subsec:integral_normalization} and Corollary \ref{cor:control_ordinary}.
\end{rem}
The pairing itself is not expected to be $p$-integral unless $k -2 < p$ since it involves $(k-2)!$ in the denominator. In other words, we have
$$\langle -, - \rangle_k : L_k(\mathcal{O}) \times L_k(\mathrm{det}^{2-k})(\mathcal{O}) \to \frac{1}{(k-2)!}\mathcal{O} .$$
\subsection{Distributions (as coefficient modules)} \label{subsec:distributions}
In order to introduce the $p$-adic deformation of quaternionic forms, we record the standard notion of $p$-adic distributions and fix convention here. Since distribution modules themselves are independent of types of $\mathbb{Z}_p$-extensions, the argument of \cite[$\S$3.3 and 3.4]{pollack-stevens} applies to our setting directly. See \cite[$\S$3]{pollack-stevens} for detail.

Let $f(z)$ be a rigid analytic or locally analytic function on $\mathbb{Z}_p$.
It admits right weight $k$ action of $\Sigma_0(p)$ as follows :
\begin{align} \label{eqn:action_on_functions}
\left( f \mid_k r \right) (z) & :=  (cz+d)^{k-2} \cdot f \left( \dfrac{az+b}{cz+d} \right) .
\end{align}

Let $\mathbf{D}_k$/$\mathscr{D}_k$ be the module of $E$-valued rigid analytic/locally analytic distributions on $\mathbb{Z}_p$ with \textbf{left} weight $k$ action of $\Sigma_0(p)$, respectively. 
The action is defined as dual :
\begin{align*}
\left( r \circ \mu_k \right) (f(z)) := \mu_k  ( ( f \mid_{k} r) (z))
\end{align*}
where $\mu_k \in \mathbf{D}_k$ or $\mathscr{D}_k$.
Note that there is a natural inclusion $\mathscr{D}_k \hookrightarrow \mathbf{D}_k$ as in \cite[$\S$3.1]{pollack-stevens}.
We regard the action as the representation $\widetilde{\rho}_k : \Sigma_0(p) \to \mathrm{Aut}_E(D_k)$ where $D_k = \mathbf{D}_k$ or $\mathscr{D}_k$.

The $\Sigma_0(p)$-equivariant specialization map $\mathrm{sp}_k : D_k \to L_k$ is defined by
$$\mathrm{sp}_k : \mu_k \mapsto \int_{\mathbb{Z}_p} (Y-zX)^{k-2} d\mu_k(z) $$
where $D_k = \mathbf{D}_k$ or $\mathscr{D}_k$.
We follow \cite[$\S$3.4]{pollack-stevens} for the convention of the specialization. Also, we defined the actions of $\Sigma_0(p)$ on both sides to make the map equivariant.

\begin{notation}
We omit $\rho_k$ and $\widetilde{\rho}_k$ if there is no confusion.
\end{notation}

\section{Overconvergent quaternionic forms and control theorems} \label{sec:automorphic_forms}
In this section, we define quaternionic forms and their overconvergent variants.
We prove the control theorem to compare them. We also give a completed cohomological description of quaternionic forms.

\subsection{Classical $p$-adic quaternionic forms}
\begin{defn}[Classical $p$-adic quaternionic forms] $ $
\begin{itemize}
\item A continuous function $\phi_k : \domain \to L_k(E)$ is called a \textbf{$p$-adic quaternionic form of discriminant $N^-$, level $N^+p^r$, and weight $k$} if $\phi_k$ satisfies the following transformation property:
$$\phi_k(\alpha b u) = \rho_k(u^{-1}_p) \circ \phi_k(b)$$
where $\alpha \in B^\times$ and $u \in \widehat{R}^\times_{N^+p^r}$, and $u_p$ is the $p$-part of $u$.
\item The space of such $p$-adic quaternionic forms is denoted by $S^{N^-}_k(N^+p^r, E)$ if $k \neq 2$. If $k = 2$, then $S^{N^-}_k(N^+p^r, E)$ denotes the space of $p$-adic quaternionic forms which are not constant. 
\item If one change the level structure by $U_r$, $Z_r$, or other level structures, one may easily define
$S^{N^-}_k(N^+, p^r, E)$, $S^{N^-}_k(Z_r, E)$, or spaces of quaternionic forms of various levels.
\end{itemize}
\end{defn}
\begin{rem}
Our quaternionic forms corresponds to ``$\ell$-adic modular forms" or ``$\ell$-adic avatar" (with $\ell = p$) in \cite[$\S$4.1]{chida-hsieh-p-adic-L-functions}. See $\S$\ref{subsec:complex_p-adic_quaternionic_forms} for detail.
\end{rem}
Let $\mathbb{T}^{N^-}_k(N^+p^r)_E$ be the full Hecke algebra over $E$ acting faithfully on $S^{N^-}_k(N^+p^r, E)$ and
 $\mathbb{T}^{N^-}_k(N^+,p^r)_E$ be the full Hecke algebra over $E$ acting faithfully on $S^{N^-}_k(N^+, p^r, E)$.

We compare their structures with classical modular forms.
Let $S_k(Np^r, E)^{N^-\textrm{-new}}$ or $S_k(N,p^r, E)^{N^-\textrm{-new}}$ be the $N^-$-new subspace of cuspforms of weight $k$ and level $\Gamma_0(Np^r)$ or $\Gamma_0(N) \cap \Gamma_1(p^r)$ whose Fourier coefficients lie in $E$ and $\mathbb{T}_k(N^+p^r)^{N^-\textrm{-new}}_E$ or $\mathbb{T}_k(N^+,p^r)^{N^-\textrm{-new}}_E$ be the corresponding quotient Hecke algebra, respectively.

Then the Jacquet-Langlands correspondence (over fields) shows the following relation between classical modular forms and quaternionic forms.
\begin{thm}[{\cite[$\S$3.3]{longo-vigni-control}}] \label{thm:jacquet-langlands}
There exist isomorphisms of Hecke algebras over $E$
\begin{align*}
\mathbb{T}^{N^-}_k(N^+p^r)_E & \simeq \mathbb{T}_k(N^+p^r)^{N^-\textrm{-new}}_E\\
\mathbb{T}^{N^-}_k(N^+,p^r)_E & \simeq \mathbb{T}_k(N^+,p^r)^{N^-\textrm{-new}}_E
\end{align*}
and (non-canonical) isomorphisms of Hecke modules
\begin{align*}
S^{N^-}_k(N^+p^r, E) & \simeq S_k(Np^r, E)^{N^-\textrm{-new}} \\
S^{N^-}_k(N^+,p^r, E) & \simeq S_k(N,p^r, E)^{N^-\textrm{-new}}
\end{align*}
as $\mathbb{T}^{N^-}_k(N^+p^r)_E$-modules and $\mathbb{T}^{N^-}_k(N^+, p^r)_E$-modules, respectively.
For a classical modular form $f_k$, we denote the corresponding quaternionic form by $\phi_{f_k}$.
\end{thm}
\begin{rem}[on Theorem \ref{thm:jacquet-langlands}] See also \cite[Theorem 2.4]{bertolini-darmon-hida-2007} for $r= 0, 1$ cases. From this case, one can deduce the general result without any serious difficulty as in \cite[$\S$3.3]{longo-vigni-control}.
\end{rem}

\subsection{Overconvergent quaternionic forms and control theorems}
We mainly follow \cite{pollack-stevens}. See also \cite[$\S$3]{wan-xiao-zhang} and\cite[$\S$2]{liu-wan-xiao}.
Let $D_k$ be either $\mathscr{D}_k$ or $\mathbf{D}_k$.
\begin{defn}[Overconvergent quaternionic forms] $ $
\begin{enumerate}
\item A continuous function $\Phi_k : \domain \to D_k$ is a \textbf{$D_k$-valued overconvergent quaternionic form of discriminant $N^-$, level $N^+p^r$, and weight $k$}
if $\Phi_k$ satisfies the following transformation property :
$$\Phi_k(\alpha b u) = \widetilde{\rho}_k(u^{-1}_p) \circ \Phi_k(b)$$
where $\alpha \in B^\times$ and $u \in \widehat{R}^\times_r$, and $u_p$ is the $p$-part of $u$.
\item The space of such overconvergent quaternionic forms is written as
$S^{N^-} (N^+p^r, D_k)$.
\item More generally, for a $\Sigma_0(p)$-module $A$, we similarly define $A$-valued quaternionic forms and denote the space of such forms by $S^{N^-} (N^+p^r, A)$ and its variants.
\end{enumerate}
\end{defn}

Using the specialization map as in $\S$\ref{subsec:distributions}, we give an explicit relation between classical and overconvergent quaternionic forms for the non-critical slope case.
Let $S^{(<k-1)}$ be the submodule of a Hecke module $S$ of slope $< k-1$ and $\mathbb{T}^{N^-}_k(N^+p)^{(<k-1)}_E$ be the Hecke algebra acting faithfully on the subspace of the forms of slope less than $k-1$.

\begin{thm}[Control theorem] \label{thm:control}
There exist $\mathbb{T}^{N^-}_k(N^+p)^{(<k-1)}_E$-equivariant isomorphisms
\begin{align*}
S^{N^-} (N^+p, \mathscr{D}_k)^{(<k-1)} & \simeq S^{N^-} (N^+p, \mathbf{D}_k)^{(<k-1)} \\
& \simeq S^{N^-}_k (N^+p, E)^{(<k-1)}
\end{align*}
where the first map is induced from the natural inclusion between distributions and the second map is induced from the specialization map $\mathrm{sp}_k$.
\end{thm}
For the first isomorphism, see \cite[Lemma 5.3]{pollack-stevens-critical}. We prove the second isomorphism in \S\ref{subsubsec:proof_control_thm}.

\subsection{Integral normalizations and integral refinement of control theorems} \label{subsec:integral_normalization}
We introduce an \emph{optimal} integral normalizations of classical and overconvergent quaternionic forms. These will be used for the slope zero case.
\begin{choice}[of the ``explicit" integral normalizations] Note that all the choices implies that nonvanishing of the form modulo $\varpi$.
\begin{enumerate}
\item If $\Phi_k \in S^{N^-} (N^+p, \mathscr{D}_k(\mathcal{O}))^{(0)}$, then we normalize that the values of $\Phi_k$ lie in $\mathscr{D}_k(\mathcal{O})$ but not in $\varpi \mathscr{D}_k(\mathcal{O})$.
\item If $\phi_k \in S^{N^-}_k (N^+p, \mathcal{O})^{(0)}$, then we normalize that the values of $\phi_k$ lie in $L_k(\mathcal{O}) = \mathrm{Sym}^{k-2}(\mathcal{O})$ but not in $\varpi \mathrm{Sym}^{k-2}(\mathcal{O})$.
\end{enumerate}
Choice (1) and (2) are compatible under the specialization map $\mathrm{sp}_k$.
\end{choice}

Let $S^{(0)}$ be the slope zero submodule of a Hecke module $S$ and $\mathbb{T}^{N^-}_k(N^+p)^{(0)}$ be the slope zero quotient of $\mathbb{T}^{N^-}_k(N^+p)$.
\begin{cor}[Integral refinement of Theorem \ref{thm:control}] \label{cor:control_ordinary}
There exists a $\mathbb{T}^{N^-}_k(N^+p)^{(0)}$-equivariant isomorphism
$$S^{N^-} (N^+p, \mathscr{D}_k(\mathcal{O}))^{(0)} \simeq S^{N^-}_k (N^+p, \mathcal{O})^{(0)}.$$
\end{cor}
\begin{rem}
This is a quaternionic analogue of \cite[Lemma 6.4]{explicit-hida}. Note that this explicit integral normalization is \emph{not} compatible with the integral normalization for Hida theory. See \cite[Lemma 6.3]{explicit-hida} for the other integral normalization, which is more relevant to the integral Hida theory.
We call the other normalization by  the ``canonical" integral normalization. These two integral normalizations coincide if $k-1 < p$. See the proof of \cite[Theorem 6.8]{explicit-hida}.
\end{rem}

\subsection{A cohomological interpretation} \label{subsec:cohomological_interpretation}
We give a cohomological interpretation of the space of quaternionic forms adapting the approach of completed cohomology \`{a} la Emerton with the ``trivial" spectral sequence. See \cite[(3.2)]{emerton-interpolation}.
We expect that the explicit Gross points plays the role of the functional on the completed cohomology whose values are (an half of) anticyclotomic $p$-adic $L$-functions as the cycle $(0) - (i\infty) \in \mathrm{Div}^0(\mathbb{P}^1(\mathbb{Q}))$ plays the same role on the completed cohomology for $\mathrm{GL}_{2 /\mathbb{Q}}$ to produce cyclotomic $p$-adic $L$-functions. This idea comes from Emerton's comment when the author gave a talk at University of Chicago.

We recall the $\Gamma_0(p^\infty)$-variant of completed cohomology for quaternion algebras.
\begin{defn}[Completed cohomology for quaternion algebras]
$$\widetilde{\mathrm{H}}^0(N^+p^\infty)  := \varprojlim_{s} \varinjlim_{r} \mathrm{H}^0(B^\times \backslash \widehat{B}^\times / \widehat{R}^{\times}_{N^+p^r}, \mathcal{O}/\varpi^s\mathcal{O}) .$$
\end{defn}

Let $L_{k,\mathcal{O}} = L_k(\mathcal{O}) = \mathrm{Sym}^{k-2}(\mathcal{O}^2)$ equipped with a continuous action of an open subgroup $(R_r \otimes_{\mathbb{Z}}\mathbb{Z}_p)^\times$ of $B^\times_p \simeq \mathrm{GL}_2(\mathbb{Q}_p)$.
We define the associated $p$-adic local system $\mathscr{L}_{k,\mathcal{O}} = \mathscr{L}_k (\mathcal{O})$ on Hida variety $X^{N^-}_{N^+p^r}$ as follows:
$$\mathscr{L}_{k,\mathcal{O}} := \mathscr{L}_k (\mathcal{O}) := B^\times \backslash \left(  \left(  \widehat{B}^\times /\widehat{R}^{(p), \times}_{N^+p^r}    \right) \times L_{k,\mathcal{O}}  \right) / R^\times_{N^+p^r,p}$$ 
For a more detailed description of the local system, see \cite[Definition 2.2.3]{emerton-interpolation} and \cite[2.1.3]{emerton-icm}.
Then the \emph{trivial} Hochschild-Serre spectral sequence shows
$$\mathrm{Hom}_{R^{\times}_{N^+p^r,p}} (L^{\vee}_{k,\mathcal{O}}, \widetilde{\mathrm{H}}^0) \simeq \mathrm{H}^0(B^\times \backslash \widehat{B}^\times / \widehat{R}^{\times}_{N^+p^r}, \mathscr{L}_{k,\mathcal{O}})$$
where $L^{\vee}_{k,\mathcal{O}} = \mathrm{Hom}(L_{k,\mathcal{O}}, \mathcal{O})$ as in \cite[2.1.3]{emerton-icm}.
Dualizing the first term, we have
$$\mathrm{Hom}_{R^\times_{N^+p^r,p}} (L^{\vee}_{k,\mathcal{O}}, \widetilde{\mathrm{H}}^0)  \simeq \mathrm{Hom}_{R^{\times}_{N^+p^r,p}} (\widetilde{\mathrm{H}}_0, L_{k,\mathcal{O}}) = \mathrm{H}^0( R^\times_{N^+p^r,p}, \mathrm{Hom}_{\mathcal{O}} (\widetilde{\mathrm{H}}_0, L_{k,\mathcal{O}}) ) = S^{N^-}_k(N^+p^r, \mathcal{O})$$
where $\widetilde{\mathrm{H}}_0$ is the $\mathcal{O}$-dual of $\widetilde{\mathrm{H}}^0$.
The first isomorphism comes from the fact $L_{k,\mathcal{O}}$ is a torsion-free $\mathcal{O}$-module of finite rank and $\widetilde{\mathrm{H}}_0$ is also a torsion-free $\mathcal{O}$-module.

\subsection{Proof of control theorems} \label{sec:control_proof}
The goal of this subsection is to prove Theorem \ref{thm:control} and Corollary \ref{cor:control_ordinary}. We follow the strategy of M.~Greenberg \cite[$\S$4]{greenberg-israel} \emph{very closely}, which studies the case of modular symbols. Note that the proof is almost identical due to Greenberg's ``geometry free" approach. Another virtue of this approach is that it is easy to see the integral nature for the slope zero subspace (Corollary \ref{cor:control_ordinary}). See also \cite[Proposition 4]{buzzard-families} for another proof. Buzzard's approach seems more adaptable with the setting of eigenvarieties as in \cite{buzzard-eigenvarieties}.

\subsubsection{Preliminaries on distribution modules}
Let $\mathbf{D}_k(\mathcal{O}) = \lbrace \mu \in \mathbf{D}_k : \mu(z^n) \in \mathcal{O} \textrm{ for all } n \geq 0  \rbrace$ and it is known that $\mathbf{D}_k(\mathcal{O})$ is a $\Sigma_0(p)$-stable submodule of $\mathbf{D}_k$.

\begin{lem}[{\cite[Lemma 1]{greenberg-israel}}]
Let $\mu \in \mathbf{D}_k$. Then the moments $\mu(z^n)$ of $\mu$ are uniformly bounded. Consequently, we have
$$\mathbf{D}_k \simeq \mathbf{D}_k(\mathcal{O}) \otimes_{\mathcal{O}} E .$$
\end{lem}

Define the filtration of $\mathbf{D}_k(\mathcal{O})$ as follows:
\begin{align*}
\mathrm{Fil}^0 \mathbf{D}_k(\mathcal{O}) & := \lbrace \mu \in \mathbf{D}_k(\mathcal{O}) : \mu(z^i) = 0 \textrm{ for all } i =0, \cdots k-2 \rbrace \\
\mathrm{Fil}^m \mathbf{D}_k(\mathcal{O}) & := \lbrace \mu \in \mathrm{Fil}^0 \mathbf{D}_k(\mathcal{O}) : \mu(z^{k-2+j} ) \in \varpi^{m-j+1} \mathcal{O} \textrm{ for all } j =1, \cdots m \rbrace
\end{align*}
for $m \geq 1$.

\begin{lem}[{\cite[Lemma 2]{greenberg-israel}}]
For each $m \geq 0$, the submodule $\mathrm{Fil}^m \mathbf{D}_k(\mathcal{O})$ is $\Sigma_0(p)$-stable.
\end{lem}

Consider the quotients
$$A^m \mathbf{D}_k(\mathcal{O}) := \mathbf{D}_k(\mathcal{O}) / \mathrm{Fil}^m \mathbf{D}_k(\mathcal{O})$$
for all $m \geq 0$.
We call $A^m \mathbf{D}_k(\mathcal{O})$ the \textbf{$m$-th approximation to the module $\mathbf{D}_k(\mathcal{O})$}.
Note that $L_k(\mathcal{O}) \simeq A^0\mathbf{D}_k(\mathcal{O})$.

\begin{lem}[{\cite[Lemma 8]{greenberg-israel}}] \label{lem:moments_invert_p}
$$S^{N^-}(N^+p, \mathbf{D}_k) \simeq S^{N^-}(N^+p, \mathbf{D}_k(\mathcal{O})) \otimes_{\mathcal{O}} E$$
\end{lem}

Let $m \in L_k$ and let $\mu$ be the unique preimage of $m$ under $\mathrm{sp}_k$ satisfying $\mu(z^j) = 0$ for $j  > k-2$. We define the \textbf{$j$-th moment of $m$} by $m(z^j) := \mu(z^j)$.

The specialization map $\mathrm{sp}_k$ naturally induces the Hecke-equivariant map between the spaces of quaternionic forms
$$\mathrm{sp}_{k,*} : S^{N^-}(N^+p, \mathbf{D}_k) \to S^{N^-}_k(N^+p, E)$$

Consider two natural projections
\[
\xymatrix{
\mathrm{sp}^{m}_k : \mathbf{D}_k(\mathcal{O}) \to A^m \mathbf{D}_k(\mathcal{O}) & \mathrm{sp}^{m+1,m}_k : A^{m+1} \mathbf{D}_k(\mathcal{O}) \to A^m \mathbf{D}_k(\mathcal{O}) .
}
\]
Then the induced maps $\mathrm{sp}^{m}_{k,*}$ and $\mathrm{sp}^{m+1,m}_{k,*}$ on the spaces of quaternionic forms are also $U_p$-equivariant.

\subsubsection{Lifting and control theorems: Proof of Theorem \ref{thm:control} and Corollary \ref{cor:control_ordinary}} \label{subsubsec:proof_control_thm}
Most argument does not concern domain, so the proofs are identical for the case of modular symbols except the convention of group actions.

Let $\alpha = \alpha_p(f_k)$ and $h = \mathrm{ord}_p(\alpha)$ for convenience.
Set
$$L^{\alpha}_k (\mathcal{O}) := \lbrace m \in L_k(\mathcal{O}) : m(z^i) \in \varpi^{h - e\cdot i}\mathcal{O}, 0 \leq i \leq \lfloor h/e \rfloor\rbrace,$$
where $e$ is the ramification index of $E / \mathbb{Q}_p$ and $\lfloor \cdot \rfloor$ is the floor function.
\begin{rem} $ $ \label{rem:ordinary_integrality}
\begin{enumerate}
\item It is known that $L^{\alpha}_k (\mathcal{O})$ is $\Sigma_0(p)$-stable.
\item $L^{\alpha}_k (\mathcal{O}) = L_k (\mathcal{O})$ if $h=0$, i.e.~slope zero.
\end{enumerate}
\end{rem}
Note that $L^{\alpha}_k (\mathcal{O}) \otimes_\mathcal{O} E = L_k (E)$ for any $\alpha$.

Let $\phi^0 \in S^{N^-}_{2} (N^+p, E)$ be an eigenform with $U_p$-eigenvalue $\alpha$ in $E$ of slope strictly less than $k-1$. Assume that $\phi^0$ is normalized, i.e.~$\phi^0 \in S^{N^-}_2 (N^+p, \mathcal{O})$.
\begin{lem}[{\cite[Lemma 11]{greenberg-israel}}] \label{lem:lemma_for_convergence}
\begin{enumerate}
\item Let $\mu \in \mathbf{D}_k(\mathcal{O})$ be such that $\mathrm{sp}_k (\mu) \in L^{\alpha}_k (\mathcal{O})$. Then
$$\left( \begin{smallmatrix} p & a \\ 0 & 1 \end{smallmatrix} \right) \circ \mu \in \alpha \mathbf{D}_k(\mathcal{O}).$$
\item Let $\mu \in \mathrm{Fil}^m \mathbf{D}_k(\mathcal{O})$. Then
$$\left( \begin{smallmatrix} p & a \\ 0 & 1 \end{smallmatrix} \right) \circ \mu \in \alpha \mathrm{Fil}^{m+1} \mathbf{D}_k(\mathcal{O}).$$
\end{enumerate}
\end{lem}

Assume the existence of a lift $\phi^m$ of $\phi^0$ to $S^{N^-} (N^+p, A^m\mathbf{D}_k(\mathcal{O}))$ such that $\phi^m$ is also a $U_p$-eigenform with eigenvalue $\alpha$. Choose an arbitrary lift $\Phi$ of $\phi^m$ to an element of $\mathrm{Hom} (B^\times \backslash \widehat{B}^\times / \widehat{R}^{(p), \times}, \mathbf{D}_k(\mathcal{O}))$.
Since $\Phi$ is also a lift of $\phi^0$, Lemma \ref{lem:lemma_for_convergence}.(1) implies that
$$\left( \frac{1}{\alpha} \cdot U_p \right) \Phi \in \mathrm{Hom} (\mathbb{Z}[B^\times \backslash \widehat{B}^\times / \widehat{R}^{(p), \times}], \mathbf{D}_k(\mathcal{O})).$$
Now we define the one step lifting
$\phi^{m+1}$ by
$$\phi^{m+1} := \mathrm{sp}^{m+1}_{k,*} ( \left( \frac{1}{\alpha} \cdot U_p \right) \Phi ) \in \mathrm{Hom} (\mathbb{Z}[B^\times \backslash \widehat{B}^\times / \widehat{R}^{(p), \times}], A^{m+1}\mathbf{D}_k(\mathcal{O})) .$$
The $U_p$-equivariance of the projection maps together with the relation
$\mathrm{sp}^m_k = \mathrm{sp}^{m+1, m}_k \circ \mathrm{sp}^{m+1}_k$
implies that
$\phi^{m} = \mathrm{sp}^{m+1, m}_{k,*} ( \phi^{m+1} ) $.
\begin{lem}[{\cite[Claim 1]{greenberg-israel}}] \label{lem:independent}
The lifted form $\phi^{m+1} \in \mathrm{Hom} (\mathbb{Z}[B^\times \backslash \widehat{B}^\times / \widehat{R}^{(p), \times}], A^{m+1}\mathbf{D}_k(\mathcal{O}))$ is independent of the choice of lift $\Phi$ used in the construction.
\end{lem}
\begin{proof}
The claim immediately follows from Lemma \ref{lem:lemma_for_convergence}.(2).
\end{proof}
\begin{rem}
We do not need to have an analogue of \cite[Claim 2]{greenberg-israel} since it concerns the special property of the fundamental domain of modular symbols.
\end{rem}
The following lemma says that the lift is also $R^\times_{N^+p,p}$-equivariant.
\begin{lem}[{\cite[Claim 3]{greenberg-israel}}]   \label{lem:U_p-equivariance_of_lifting}
The lifted form $\phi^{m+1}$ is $R^\times_{N^+p,p}$-invariant. In other words,
$$\gamma \circ \phi^{m+1} = \phi^{m+1} .$$
where $\gamma \in R^\times_{N^+p,p}$.
\end{lem}
\begin{proof}
It is a standard computation with help of Lemma \ref{lem:independent}.
\end{proof}
%
The following lemma directly follows from the $\Sigma_0(p)$-equivariance of $\mathrm{sp}^{m+1}_{k,*}$.
\begin{lem} \label{lem:eigenvalues}
$\phi^{m+1}$ is a $U_p$-eigenform with eigenvalue $\alpha$.
\end{lem}
Lemma \ref{lem:independent}, Lemma \ref{lem:U_p-equivariance_of_lifting}, and Lemma \ref{lem:eigenvalues} directly imply the following proposition.
\begin{prop}[{\cite[Proposition 12]{greenberg-israel}}]\label{prop:lifting}
The lifted form $\phi^{m+1} \in S^{N^-} (N^+p, A^{m+1}\mathbf{D}_k(\mathcal{O}))$ is well-defined and independent of the choice of lift $\Phi$ used in the construction. Moreover, $U_p \phi^{m+1} = \alpha \phi^{m+1}$ in $S^{N^-} (N^+p, A^{m+1}\mathbf{D}_k(\mathcal{O}))$.
\end{prop}
In order prove Theorem \ref{thm:control}, it suffices to prove 
$$ S^{N^-} (N^+p, \mathbf{D}_k)^{U_p = \alpha} \simeq S^{N^-}_{k} (N^+p, E)^{U_p = \alpha} $$
for each $\alpha$ with $\mathrm{ord}_p(\alpha) < k-1$. Also, due to Remark \ref{rem:ordinary_integrality}.(2), the following theorem implies Corollary \ref{cor:control_ordinary} immediately.
\begin{thm}[Analogue of {\cite[Theorem 9]{greenberg-israel}}]
Let $\alpha \in E$ be an $U_p$-eigenvalue acting on $S^{N^-}_k (N^+p, E)$ with noncritical slope $h = \mathrm{ord}_p(\alpha) < k-1$.
Then the specialization map induces an Hecke-equivariant isomorphism
$$\mathrm{sp}_{k,*} : S^{N^-}(N^+p, \mathbf{D}_k(E))^{U_p = \alpha} \to S^{N^-}_k(N^+p, E)^{U_p = \alpha}.$$
\end{thm}
\begin{proof}
A proof can be taken verbatim from that of \cite[Theorem 9]{greenberg-israel}. We just remark that
Lemma \ref{lem:moments_invert_p} and Lemma \ref{lem:lemma_for_convergence}.(2) are used to prove the injectivity, and 
Proposition \ref{prop:lifting} is used to prove the surjectivity.
\end{proof}

\section{Construction of $p$-adic $L$-functions} \label{sec:construction}
The goal of this section is to give an overconvergent construction of anticyclotomic $p$-adic $L$-functions as admissible distributions and reconstruct the corresponding theta elements from the distributions.

Let $h < k-1$. We briefly recall the $h$-admissibility of distributions on locally polynomials on $\mathbb{Z}_p$ and the work of Amice-Velu \cite{amice-velu} and Visik \cite{vishik} (Theorem \ref{thm:amice-velu}) on the lifting $h$-admissible distributions on locally polynomial functions on $\Gamma_\infty \simeq \mathbb{Z}_p$ of degree $\leq k-2$ to locally analytic distributions on $\Gamma_\infty$.
Then we explicitly define the $h$-admissible distribution on locally polynomial functions on $\Gamma_\infty$ of degree $\leq k-2$ in terms of the values of quaternionic forms.
Applying the lifting result, the distribution extends to a locally analytic distribution on  $\Gamma_\infty$.

\subsubsection{Preliminaries on distributions}
We recall the unique lifting of $h$-admissible distributions on locally polynomial functions of degree $\leq k-2$ to locally analytic distributions. See \cite[$\S$1.3]{vishik} and \cite[$\S$2.1]{pollack-thesis} for detail.

Let $\mathcal{C}^{h}(\mathbb{Z}_p)$ be the space of $\mathbb{C}_p$-valued functions on $\mathbb{Z}_p$ which are locally polynomials of degree $\leq h$.
\begin{defn}[$h$-admissible distributions on locally polynomials] \label{defn:admissible_distributions}
A \textbf{$h$-admissible distribution $\mu$ on $\mathbb{Z}_p$} is a $\mathbb{C}_p$-linear map from $\mathcal{C}^{h}(\mathbb{Z}_p)$ to $\mathbb{C}_p$ such that
$$\sup_{a \in \mathbb{Z}_p} \left\vert \mu \left( (z-a)^i \cdot \mathbf{1}_{a+p^n\mathbb{Z}_p}(z) \right) \right\vert$$
is $O(p^{n(h-i)})$ for $0 \leq i \leq h$.
\end{defn}

Let $\mathscr{D}$ be the algebra of locally analytic distributions on $\mathbb{Z}_p$ with convolution product $*$ but forgetting the weight $k$ action of $\Sigma_0(p)$.
\begin{thm}[Amice-Velu, Visik] \label{thm:amice-velu}
Let $\mu$ be a $h$-admissible distribution on locally polynomial functions on $\mathbb{Z}_p$ of degree less than or equal to $k-2$. Then $\mu$ extends uniquely to a distribution on locally analytic functions on $\mathbb{Z}_p$, i.e. $\mu \in \mathscr{D}$.
\end{thm}
\begin{proof}
See \cite[Proposition II.2.4]{amice-velu}, \cite[Lemma 2.10 and Theorem 3.3]{vishik}, \cite[Theorem in $\S$11]{mtt}, and \cite[(6.5) Corollary]{glenn-oms} for detail. Note that the original statement is given with distributions on $\mathbb{Z}^\times_p$ rather than on $\mathbb{Z}_p$. See \cite[$\S$1.8 and $\S$2.4]{vishik} for the modification.
\end{proof}

\subsection{The distribution} \label{subsec:the_distribution}
Let $f_k \in S_k(\Gamma_0(Np))$ be a $p$-stabilized newform of non-critical slope and $\phi_{f_k}$ be the associated integrally normalized quaternionic form in $S^{N^-}_k(N^+p, \mathcal{O})^{(<k-1)}$ 
as in $\S$\ref{subsec:integral_normalization}.

\begin{prop}[{\cite[Lemma 6.2]{pollack-stevens}}] \label{prop:admissibility}
All the values of $\Phi_{f_k}$ are $(k-1)$-admissible distributions.
\end{prop}
From now on, we explicitly determine the distribution $\Phi_{f_k} (\varsigma^{(1)})$, which is an half of the $p$-adic $L$-function.
\begin{defn}[The distribution] \label{def:the_distribution}
$$\mu_{f_k, K_\infty} := \Phi_{f_k} (\varsigma^{(1)}) .$$
\end{defn}
\begin{notation}
We fix notation:
\[
\xymatrix@R=0em{
\widehat{K}^\times \ar@{->>}[r] & \widetilde{G}_\infty \ar@{->>}[r] & \Gamma_\infty \ar[r]^-{\simeq} & \mathbb{Z}_p \\
\xi \ar@{|->}[rr] & & \xi_a \ar@{|->}[r] & a
}
\]
so that $\xi_a W_n = a+ p^n\mathbb{Z}_p$ under the last identification. This one also yields the following equality at the level of twisted explicit Gross points:
$$x_n(\xi_a) = \varsigma^{(1)}  \cdot \left( \begin{matrix} p^n & a \\ 0 & 1\end{matrix}  \right) .$$
\end{notation}

Due to Theorem \ref{thm:amice-velu} and Proposition \ref{prop:admissibility}, in order to determine the distribution explicitly, 
it suffices to compute the values
$$\Phi_{f_{k}}  ( \varsigma^{(1)} ) ( z^j \cdot \mathbf{1}_{\xi_a W_n} )$$
 explicitly for all $n \geq 1$, $a \in \mathbb{Z}/p^n\mathbb{Z}$ and $j = 0, \cdots, k-2$.

\begin{defn}[The $j$-th component of an quaternionic form]
We define the \textbf{$j$-th component $\phi^{[j]}_{f_k}$ of $\phi_{f_k}$} by the composition
\[
\xymatrix{
\domain \ar[r]^-{\phi_{f_k}} \ar@/_2pc/[rrr]_-{\phi^{[j]}_{f_k}} & L_k(\mathcal{O}) \ar[rr]^-{\langle -,X^{k-2-j}Y^{j}\rangle_k} & & E 
}
\]
for $j = 0, \cdots, k-2$.
\end{defn}

First, we compare the evaluation of overconvergent quaternionic forms and the specialization map, which describe the ``total measure."
\begin{lem}[on the comparison of the total measure] \label{lem:total_measure}
Let $b \in \domain$.
Then 
$$\Phi_{f_{k}} ( b ) (z^j) = \phi^{[j]}_{f_k} (b)$$
for $j = 0, \cdots k-2$.
\end{lem}
\begin{proof}
Recall the specialization map
$$\mathrm{sp}_k \left( \Phi_{f_{k}} ( b )  \right) = \int_{\mathbb{Z}_p} (Y -zX)^{k-2} d \Phi_{f_{k}} (b) (z)  .$$
Also, the control theorem (Theorem \ref{thm:control}) implies that
\begin{align*}
\mathrm{sp}_{k,*} \left( \Phi_{f_{k}} \right) ( b )  & = \mathrm{sp}_k \left( \Phi_{f_{k}} (b)  \right) \\
& = \phi_{f_{k}} (b) .
\end{align*}
Thus, we have
$$\phi_{f_{k}} (b) = \int_{\mathbb{Z}_p} (Y -zX)^{k-2} d \Phi_{f_{k}} (b) (z) \in L_k.$$
Pairing with $\langle - , X^{k-2-j} Y^j \rangle_k$ as in $\S$\ref{subsubsec:pairing}, we have the following equality of the total measure.
$$\int_{\mathbb{Z}_p} z^j d \Phi_{f_k} (b) (z) = \phi^{[j]}_{f_k} (b)$$
for all $j = 0, \cdots, k-2$.
\end{proof}

Now we compute all the values following \cite[Proposition 6.3]{pollack-stevens}.
Since $\Phi_{f_k}$ is an $U_p$-eigenform with eigenvalue $\alpha$, we have
\begin{align*}
\Phi_{f_{k}} ( \varsigma^{(1)} ) & = \alpha^{-n} \left( U^n_p \Phi_{f_{k}} \right) ( \varsigma^{(1)} ) \\
& = \alpha^{-n} \sum_{b=0}^{p^n-1}  \left( \begin{smallmatrix} p^n & b \\ 0 & 1\end{smallmatrix}  \right) \circ \left( \Phi_{f_{k}}  ( \varsigma^{(1)} \cdot \left( \begin{smallmatrix} p^n & b \\ 0 & 1\end{smallmatrix}  \right) ) \right) \\
& = \alpha^{-n} \sum_{b=0}^{p^n-1}  \left( \begin{smallmatrix} p^n & b \\ 0 & 1\end{smallmatrix}  \right) \circ \left( \Phi_{f_{k}}  ( x_n(\xi_b) ) \right) .
\end{align*}
For any distribution $\mu \in \mathscr{D}_k$, the support of $\left( \begin{smallmatrix} p^n & a \\ 0 & 1\end{smallmatrix}  \right) \circ \mu$ is contained in $a + p^n \mathbb{Z}_p$.
Thus, for $0 \leq j \leq k-2$, we have
\begin{align*}
\Phi_{f_k} (\varsigma^{(1)}) (z^j \cdot \mathbf{1}_{a+p^n\mathbb{Z}_p}(z)) & = \alpha^{-n} \cdot \left( U^n_p \Phi_{f_k} \right) ( \varsigma^{(1)}  ) (z^j \cdot \mathbf{1}_{a+p^n\mathbb{Z}_p}(z)) \\
& = \alpha^{-n} \cdot \sum_{b=0}^{p^n-1}  \left( \begin{smallmatrix} p^n & b \\ 0 & 1\end{smallmatrix}  \right) \circ \left( \Phi_{f_k}  ( x_n(\xi_b) ) \right) (z^j \cdot \mathbf{1}_{a+p^n\mathbb{Z}_p}(z)) \\
& = \alpha^{-n} \cdot \sum_{b=0}^{p^n-1}  \left( \Phi_{f_k}  ( x_n(\xi_b) ) \right) ( (p^n z+b)^j \cdot \mathbf{1}_{a+p^n\mathbb{Z}_p}(p^n z+b)) \\
& = \alpha^{-n} \cdot   \Phi_{f_k}  ( x_n(\xi_a))  ( (p^n z+a)^j ) \\
& = \alpha^{-n} \cdot \sum_{i=0}^{j}  \Phi_{f_k}   ( x_n(\xi_a))  \left( \left( {\begin{matrix} j \\ i \end{matrix}} \right) \cdot (p^n z )^i \cdot a^{j-i} \right)  \\
& =  \sum_{i=0}^{j} \left( \frac{p^i}{\alpha} \right)^{n} \cdot  \phi^{[i]}_{f_k}   ( x_n(\xi_a))  \cdot a^{j-i}  . & \textrm{Lemma \ref{lem:total_measure}} \\
\end{align*}
Thus, the distribution $\Phi_{f_k}(\varsigma^{(1)})$ is completely determined.
Also, for $j \geq 1$, we have
\begin{equation} \label{eqn:overconvergent_classical_comparison}
\Phi_{f_k} (\varsigma^{(1)}) \left( \left( z^j - a \cdot z^{j-1} \right) \cdot \mathbf{1}_{a+p^n\mathbb{Z}_p}(z) \right)
= \left( \frac{p^j}{\alpha} \right)^{n} \cdot  \phi^{[j]}_{f_k}   ( x_n(\xi_a))  .
\end{equation}

Let $\mu^{-1}_{f_k, K_\infty}$ is the distribution in $\mathscr{D}$ determined by the values
$$\mu^{-1}_{f_k, K_\infty} (z^j \cdot \mathbf{1}_{a+p^n\mathbb{Z}_p})
= \sum_{i=0}^{j} \left( \frac{p^i}{\alpha} \right)^{n} \cdot  \phi^{[i]}_{f_k}   ( x^{-1}
_n(\xi_{a}))  \cdot a^{j-i} .$$
\begin{defn}[$p$-adic $L$-functions] \label{defn:p-adic_L-functions}
The \textbf{$p$-adic $L$-function $L_p(K_\infty, f)$} is defined by the convolution product of distributions
$$\mu_{f_k, K_\infty} * \mu^{-1}_{f_k, K_\infty} \in  \mathscr{D} .$$
\end{defn}
\begin{rem}
Here, the convolution product can be understood as a $p$-adic ``monodromy pairing" or ``height pairing" in the spirit of Gross or Gross-Zagier formula. 
This element is well-defined up to a nonzero constant in $\mathcal{O}^\times_E$ since all the choices defining $\mu_{f_k, K_{\infty}}$ and $\mu^{-1}_{f_k, K_{\infty}}$ cancel each other.
\end{rem}

\subsection{Reconstruction of theta elements}
In order to obtain the interpolation formula more explicitly, we compare our $p$-adic $L$-functions and those of \cite{chida-hsieh-p-adic-L-functions} at the level of theta elements (finite layers). 
\begin{defn}[Theta elements]
Let
$$\widetilde{\theta}_n( f_k)  :=  \sum_{\xi \in \widetilde{G}_n} \left( \frac{1}{\alpha}\right)^n \cdot  \phi^{[0]}_{f_k} ( x_n(\xi)) \xi^{-1} \in E[\widetilde{G}_n]$$
and
$$\widetilde{\theta}^{*}_n( f_k)   :=  \sum_{\xi \in \widetilde{G}_n} \left( \frac{1}{\alpha}\right)^n \cdot  \phi^{[0]}_{f_k} ( x^{-1}_n(\xi)) \xi^{-1} \in E[\widetilde{G}_n].$$
We define the \textbf{$n$-th theta elements of $f_k$} by
$$\theta_n( f_k)  :=  \textrm{the image of } \theta(\widetilde{K}_n, f_k) \textrm{ under the projection } E[\widetilde{G}_n] \to E[\Gamma_n]$$
and
$$\theta^*_n( f_k)  :=  \textrm{the image of } \theta^*(\widetilde{K}_n, f_k) \textrm{ under the projection } E[\widetilde{G}_n] \to E[\Gamma_n] .$$
\end{defn}

\begin{rem}[on the well-definedness]
Each element is defined only up to multiplication by an element of $\widetilde{G}_n$ due to the choices we made.
The element 
$$L_{p,n}(K_\infty, f_k ):= \theta_n( f_k) \cdot \theta^*_n( f_k) \in E[\Gamma_n]$$
 is only well-defined.
\end{rem}
\begin{rem}[on the boundedness of coefficients]
By construction, it is easy to observe that
$$\theta_n( f_k) \in \frac{1}{\alpha^n}\mathcal{O}_E[\Gamma_n].$$
(cf. \cite[Remark 2.5 and Definition 2.6]{castella-longo}, \cite[Lemma 4.4.(1)]{chida-hsieh-p-adic-L-functions}.)
\end{rem}

\section{The weak interpolation formula} \label{sec:weak_interpolation}
The goal of this section to prove the ``weak" interpolation formula for our $p$-adic $L$-functions, indeed $p$-adic theta elements. We use the interpolation formula for complex theta elements in \cite{chida-hsieh-p-adic-L-functions}.
Since our $p$-adic theta elements and complex theta elements of \cite{chida-hsieh-p-adic-L-functions} are only congruent modulo $p^n$ at explicit Gross points of conductor $p^n$ (Corollary \ref{cor:higher_weight_theta_congruences}), the interpolation formula is given only as a congruence formula unless the form is ordinary or of weight two.
\begin{rem} \label{rem:normalization_comparison_with_CH}
All the normalizations are \emph{slightly different} from those of \cite{chida-hsieh-p-adic-L-functions} mainly due to the normalization of the pairing. Also, the index is also \emph{different} because \cite{chida-hsieh-p-adic-L-functions} mainly focus on the \emph{central} critical $L$-values and we mainly concern the growth of the distribution we defined. 
\end{rem}

\subsection{Complex quaternionic forms and $p$-adic quaternionic forms} \label{subsec:complex_p-adic_quaternionic_forms}
Using $i_\mathbb{C}$, we define the representation
$$\rho_{k,\infty}: B(\mathbb{R})^\times \to \mathrm{GL}_2(\mathbb{C}) \to \mathrm{Aut}_{\mathbb{C}}(L_k(\mathbb{C})) .$$
Then $\mathbb{C} \cdot X^{i}Y^{k-2-i} \subseteq L_k(\mathbb{C})$ or $L_k(2-k)(\mathbb{C})$ is the eigenspace on which 
$\rho_{k, \infty}(t)$ with eigenvalue $\overline{t}^i \cdot t^{k-2-i}$
or
$\rho^*_{k, \infty}(t)$ with eigenvalue $\overline{t}^{i-(k-2)} \cdot t^{-i}$
 for $t \in (K \otimes \mathbb{R})^\times$. (cf.~\cite[$\S$2.3]{chida-hsieh-p-adic-L-functions}.)

Let $U$ be an open compact subgroup of $\widehat{B}^\times$.
\begin{defn}[Complex quaternionic forms] \label{def:complex_quaternionic_forms}
A function $\mathbf{f}_k : \widehat{B}^\times \to  L_k(\mathbb{C})$ is a \textbf{complex quaternionic form of weight $k$ and level $U$} if $\mathbf{f}_k$ satisfies the transformation property
$$\mathbf{f}_k(\alpha b u) = \rho_{k, \infty}(\alpha) \circ \mathbf{f}_k(b)$$
where $\alpha \in B^\times$ and $u \in U$.
\end{defn}
The space of complex quaternionic forms is denoted by $\mathbf{S}^{N^-}_k(U, \mathbb{C})$. 
Then ${\displaystyle \mathbf{S}^{N^-}_k(\mathbb{C}) := \varinjlim_{U} \mathbf{S}^{N^-}_k(U, \mathbb{C}) }$ becomes an admissible representation of $\widehat{B}^\times$.

Let $B(\mathbb{A})^\times = ( B \otimes \mathbb{A} )^\times$ where $\mathbb{A}$ is the ring of adeles over $\mathbb{Q}$. 
For $\mathbf{f}_k \in \mathbf{S}^{N^-}_k(\mathbb{C})$ and $P(X,Y) \in L_k(2-k)(\mathbb{C})$, we define a function $\Psi(\mathbf{f}_k \otimes P(X,Y)  ) : B^\times \backslash B(\mathbb{A})^\times \to \mathbb{C}$
by
$$\Psi(\mathbf{f}_k \otimes P(X,Y)  )(g) := \left\langle \mathbf{f}_k(g_f), \rho^*_{k, \infty}(g_\infty) \circ P(X,Y) \right\rangle_k .$$
where $\rho^*_{k, \infty}  (g_\infty) = \mathrm{det}^{2-k} (g_\infty) \cdot \rho_{k, \infty}  (g_\infty)$. (cf.~\cite[(2.11)]{chida-hsieh-p-adic-L-functions}.)

We define $$\rho_{k, p} : B^\times_p \to \mathrm{Aut}_{\mathbb{C}_p} (L_k(\mathbb{C}_p))$$
by $\rho_{k, p}(g) := \rho_k \circ \iota_p \circ i_{K} (g)$ for $g \in B^\times_p$, and define
$\rho^*_{k, p}(g) := \mathrm{det}^{2-k} (g) \cdot  \rho_{k, p}(g)$.
By \cite[$\S$4.1]{chida-hsieh-p-adic-L-functions}, we have
$$\rho_{k, p} (g) = \rho_{k, \infty} (g)$$
 for $g \in B^\times$,
and
$$\rho_{k, p} (g) = \rho_k( \gamma_p \cdot i_p (g ) \cdot \gamma^{-1}_p)$$
for $g \in B^\times_p$ 
where 
$\gamma_p := \left( \begin{matrix}
\sqrt{\beta} & - \sqrt{\beta \theta} \\
-1 & \theta
\end{matrix} \right) \in \mathrm{GL}_2(K_p)$ and
$i_p$ is the fixed isomorphism $B_p \simeq \mathrm{M}_2(\mathbb{Q}_p)$.

Let $A$ be a subring of $\mathbb{C}$ and 
$\mathbf{S}^{N^-}_k(N^+p, A) \subseteq \mathbf{S}^{N^-}_k(N^+p, \mathbb{C})  $ and
$S^{N^-}_k(N^+p, A) \subseteq S^{N^-}_k(N^+p, \mathbb{C}_p)$ (via $\iota_p$ and $\iota_\infty$) 
be the submodules of $A$-valued forms.
If $\frac{1}{p} \in A$, then we have the isomorphism
\[
\xymatrix@R=0em{
\mathbf{S}^{N^-}_k(N^+p, A) \ar[r]^-{\simeq} & S^{N^-}_k(N^+p, A)  \\
\mathbf{f}_k \ar@{|->}[r] & \widehat{\mathbf{f}}_k(g) = \phi_{f_k}(g) := \rho_k(\gamma^{-1}_p) \cdot \rho_{k, p} (g^{-1}_p) \circ \mathbf{f}_k(g) 
}
\]
where $g \in \widehat{B}^\times$ and $g_p$ is the $p$-part of $g$.
Furthermore, with the Jacquet-Langlands correspondence (Theorem \ref{thm:jacquet-langlands}), we identify
\[
\xymatrix@R=0em{
\mathbf{S}^{N^-}_k(N^+p, A) \ar[r]^-{\simeq} & S^{N^-}_k(N^+p, A) & S_k(N^+p, A)^{N^-\textrm{-new}} \ar[l]_-{\simeq} \\
\mathbf{f}_k \ar@{|->}[r] & \widehat{\mathbf{f}}_k = \phi_{f_k}  & f_k \ar@{|->}[l] .
}
\]

\subsection{Complex theta elements of higher weight}
From now on, we assume $\mathbf{f}_k \in \mathbf{S}^{N^-}_k(N^+p, \mathbb{C})$ is a $p$-stabilized newform with $U_p$-eigenvalue $\alpha$ with non-critical slope. (cf. \cite[$\S$3.2]{chida-hsieh-p-adic-L-functions}.)
Let
$$v^*_i :=  (-1)^{\frac{k-2}{2}} \cdot (-1)^i \cdot \left( {\begin{matrix}
 k-2 \\ i
 \end{matrix}} \right) \cdot \sqrt{\beta}^{k-2-i} \cdot \sqrt{- D_K}^{k-2} \cdot X^{i} Y^{k-2-i} \in L_k(2-k)(\mathbb{C}) $$
for $i = 0, \cdots , k-2$. (cf. \cite[(3.1)]{chida-hsieh-p-adic-L-functions}.)

Let $$\mathbf{f}_{k,i} := \Psi(\mathbf{f}_k \otimes v^*_i  ).$$

Let $\widetilde{G}_n := K^\times \backslash \widehat{K}^\times / \mathcal{O}^\times_n \cdot \widehat{\mathbb{Q}}^\times$ be the Galois group of ring class field of $K$ of conductor $p^n$. It coincides with the same notation in $\S$\ref{subsec:classical_gross_pts}. Here, $\mathcal{O}_n = \mathbb{Z} + p^n \mathcal{O}_K$.
Let $[\cdot]_n: \widehat{K}^\times \to \widetilde{G}_n$ be the geometrically normalized reciprocity map.
\begin{defn}[Complex theta elements, {\cite[Definition 4.1]{chida-hsieh-p-adic-L-functions}}]
Fix a set $\Xi_n$ of representatives of $\widetilde{G}_n$ in $K^\times \backslash \widehat{K}^\times$. We define the \textbf{$n$-th complex theta element $\Theta^{[i]}_n (\mathbf{f}_k)$ of weight $i - \frac{k-2}{2}$} by
$$\Theta^{[i]}_n (\mathbf{f}_k) := \frac{1}{\alpha^n} \sum_{a \in \Xi_n} \mathbf{f}_{k,i}(x_n(a)) \cdot \iota^{-1}_p \left( \overline{a}^{i-(k-2)}_p \cdot a^{-i}_p \right) \cdot [a]_n  \in \mathbb{C}[\widetilde{G}_n]$$
\end{defn}
\begin{rem}
Note that the index and the weight of the complex theta element are different due to Remark \ref{rem:normalization_comparison_with_CH}. 
\end{rem}

\subsection{The interpolation formula for the complex theta elements} \label{subsec:interpolation_complex_theta}
\begin{defn} \label{def:characters}
Let $\chi$ be an anticyclotomic Hecke character of conductor $p^s$ and weight $(i - \frac{k-2}{2}, i + \frac{k-2}{2})$ where $i = 1, \cdots k-1$. We define the \textbf{(central critical twisted) $p$-adic avatar $\widehat{\chi}$ of $\chi$} by
$$\widehat{\chi}(a) = \chi(a) \cdot a^{i}_p \cdot \overline{a}^{k-2-i}_p$$
with respect to the weight $k$ of $f$. (cf.~\cite[Introduction]{chida-hsieh-p-adic-L-functions}).
\end{defn}
We state the interpolation formula for the complex theta element $\Theta^{[i]}_n (\mathbf{f}_k)$. Note that the formula is slightly different from the original one, but this is only because of the difference of the normalization.
\begin{thm}[{\cite[Proposition 4.3]{chida-hsieh-p-adic-L-functions}}] \label{thm:interpolation_complex_theta}
Suppose that $\chi$ has the conductor $p^s$. For every $n \geq \mathrm{max}\lbrace s, 1 \rbrace$, we have the interpolation formula
$$\widehat{\chi} \left( \Theta^{[i]}_n (\mathbf{f}_k)^2 \right) = (k-2)! \cdot \frac{L(f_k, \chi, \frac{k}{2})}{\Omega_{f_k,N^-}} \cdot e_p(f_k,\chi)^j \cdot \frac{1}{\alpha^{2s}} \cdot (p^s D_K)^{k-1} \cdot \frac{u^2_K}{\sqrt{D_K}} \cdot \chi(\mathfrak{N}^+) \cdot \epsilon_p(f_k) $$
where
\begin{itemize}
\item $\Omega_{f_k,N^-}$ is the Gross period defined in \cite[(4.3)]{chida-hsieh-p-adic-L-functions} (cf.\cite[Appendix]{kim-summary}),
\item $e_p(f_k, \chi)$ is the $p$-adic multiplier defined by
\begin{displaymath}
e_p(f_k,\chi) = \left\lbrace \begin{array}{ll}
1 & \textrm{if} \ n > 0, \\
\left(1 - p^{\frac{k-2}{2}} \alpha_p(f_k)^{-1}  \right)^2 & \textrm{if} \ n = 0 \ \textrm{and} \ p \ \textrm{splits in $K$}, \\
1-p^{k-2} \alpha_p(f_k)^{-2}  & \textrm{if} \ n = 0 \ \textrm{and} \ p \ \textrm{is inert in $K$,}
\end{array}
\right.
\end{displaymath}
and  
\begin{displaymath}
j = \left\lbrace \begin{array}{ll}
1 & \textrm{ if } f_k \textrm{ is new at } p, \\
2 & \textrm{ if } f_k \textrm{ is old at } p,
\end{array}
\right.
\end{displaymath}
\item $u^2_K = \# \left( \mathcal{O}^\times_K / 2 \right)$,
\item $\mathfrak{N}^+$ is the ideal of $\mathcal{O}_K$ satisfying $N^+ = \mathfrak{N}^+ \cdot \overline{\mathfrak{N}^+}$ in $K$ depending on the orientation of the optimal embedding, and
\item $\epsilon_p(f_k)$ is the eigenvalue of the Atkin-Lehner involution of $f_k$ at $p$.
\end{itemize}
\end{thm}
In \cite[Theorem A]{chida-hsieh-p-adic-L-functions}, there are certain restrictions on weight ($k-2 < p$) and slope (slope zero).
However, Theorem \ref{thm:interpolation_complex_theta} does not have such a restriction.

\subsection{An integral comparison of complex and $p$-adic quaternionic forms}
The following proposition plays the key role in the connection between complex and $p$-adic theta elements.

\begin{prop}[{\cite[Lemma 4.4]{chida-hsieh-p-adic-L-functions}}] \label{prop:congruences_at_gross_points}
Let $A \subset \mathbb{C}_p$ be a subring which contains $\mathcal{O}_{E}$ and $\mathcal{O}_{K_p}$.
For $a \in \widehat{K}^\times$, we have
\begin{enumerate}
\item $$p^{n(k-2)} \cdot \mathbf{f}_{k,i} (x_n(a)) \cdot \overline{a}^{i-(k-2)}_p \cdot a^{-i}_p  \in A$$
\item $$  p^{n( k-2)} \cdot \mathbf{f}_{k,i} (x_n(a)) \cdot \overline{a}^{i-(k-2)}_p \cdot a^{-i}_p \equiv  (-1)^{\frac{k-2}{2}} \cdot (-1)^i \cdot \left( {\begin{matrix}
 k-2 \\ i
 \end{matrix}} \right) \cdot \sqrt{\beta}^{2-k} \cdot \left\langle \phi_{f_k}(x_n(a)), X^{k-2} \right\rangle_k  \pmod{p^n A}$$
\end{enumerate}
where $a_p \in (K \otimes \mathbb{Q}_p)^\times$ is the $p$-part of $a \in \widehat{K}^\times$ and $\overline{a}_p$ is the conjugate of $a_p$.
\end{prop}
\begin{proof}
This is \cite[Lemma 4.4]{chida-hsieh-p-adic-L-functions} with change of normalization of the pairing ($\S$\ref{subsubsec:pairing}) and \cite[Remark 2.5]{castella-longo}.
\end{proof}
The following corollary is an immediate consequence.
\begin{cor} \label{cor:higher_weight_theta_congruences}
$$p^{n( k-2)} \cdot \Theta^{[i]}_{n} (\mathbf{f}_k) \equiv (-1)^{\frac{k-2}{2}} \cdot (-1)^i \cdot \left( {\begin{matrix}
 k-2 \\ i
 \end{matrix}} \right) \cdot \sqrt{\beta}^{2-k} \cdot \theta_n( f_k) \pmod{\frac{p^n}{\alpha^n_p} \mathcal{O}_E} .$$
\end{cor}
Combining Theorem \ref{thm:interpolation_complex_theta} and Corollary \ref{cor:higher_weight_theta_congruences}, we have
the following interpolation formula.
\begin{cor}[Weak interpolation formula] \label{cor:interpolation_p-adic_theta}
Let $\widehat{\chi}$ be a character as in Definition \ref{def:characters}.
\begin{align*}
& (-1)^{\frac{k-2}{2}} \cdot (-1)^i \cdot \left( {\begin{matrix}
 k-2 \\ i
 \end{matrix}} \right) \cdot \sqrt{\beta}^{2-k} \cdot \widehat{\chi} \left( \theta_n( f_k) \right) \\
  \equiv &
\sqrt{ p^{n( k-2)} \cdot (k-2)! \cdot \frac{L( f_k , \chi, \frac{k}{2})}{\Omega_{f,N^-}} \cdot e_p( f_k ,\chi)^j \cdot \frac{1}{\alpha^{2s}} \cdot (p^s D_K)^{k-1} \cdot \frac{u^2_K}{\sqrt{D_K}} \cdot \chi(\mathfrak{N}^+) \cdot \epsilon_p( f_k ) }
\pmod{\frac{p^n}{\alpha^n_p} \mathcal{O}_E} .
\end{align*}
If $f_k$ is ordinary or $k=2$, then the congruence becomes equality.
\end{cor}
\begin{rem}
If the form is ordinary, then the congruence in the weak interpolation formula becomes equality by taking the limit $n \to \infty$. Then the range of interpolating characters becomes larger, namely the set of locally algebraic $p$-adic characters of weight $(i, -i)$ with $-k/2 < i < k/2$ as in \cite[Theorem 4.6]{chida-hsieh-p-adic-L-functions}.
\end{rem}

\section{Speculations and questions} \label{sec:speculations}
\subsection{``Deformation" of explicit Gross points}
This is inspired by \cite[(4.5)]{emerton-interpolation}.
We may interpret the explicit Gross point $\varsigma^{(1)}$ as a functional on $S^{N^-}_k(N^+p, E)^{(<k-1)}$ via the overconvergent construction:
\[
\xymatrix@R=0em{
(\varsigma^{(1)})_* : S^{N^-}_k(N^+p, E)^{(<k-1)} \ar[r] & \mathscr{D}\\
\phi_{f_k} \ar[r] & \Phi_{f_k} (\varsigma^{(1)}).
}
\]
This can be regarded as a functional on a small piece (i.e.~fixed weight) of the completed cohomology as in $\S$\ref{subsec:cohomological_interpretation} or the corresponding eigenvariety.
If we can ``patch" this functional on all the weight coherently, the map $(\varsigma^{(1)})_*$ would extend to the functional on the whole completed cohomology or the corresponding eigenvariety, which may produce two variable anticyclotomic $p$-adic $L$-functions on the eigenvariety. In the sequel paper in preparation, we construct two variable anticyclotomic $p$-adic $L$-functions of Hida families and study their Iwasawa theory.

Furthermore, if we work with geometry of the eigenvariety attached to a definite quaternion algebra (e.g.~\cite{buzzard-eigenvarieties}) or uses a relevant $p$-adic extension of Jacquet-Langlands correspondence (e.g.~\cite{harris-iovita-stevens}) in more detail, then we may be able to generalize the main result of this article to the critical slope case as in \cite{bellaiche-critical}. In fact, our setting prevents CM forms; thus, it would be easier than the case of modular symbols.


\subsection{Growth of the size of noncommutative class numbers}
It seems very interesting if we see a certain Iwasawa-theoretic phenomena in the growth behavior of the size of $B^\times \backslash \widehat{B}^\times / \widehat{R}^\times_{N^+p^r}$ as $r \to \infty$.
Then the study of this growth would be regarded as a ``noncommutative Iwasawa theory" in a completely different sense.
Maybe a $p$-adic variation of the Eichler trace formula \cite[Chapitre III.$\S$5.C]{vigneras} would yield an asymptotic formula like Iwasawa's formula on the $p$-class numbers of the $p$-cyclotomic fields.

\subsection{Geometric aspect of quaternionic Hida theory}
In \cite[$\S$2 and $\S$6]{longo-vigni-manuscripta}, Longo and Vigni investigate the geometric aspect of quaternionic Hida theory. If we can capture
the $\Gamma_1(p^r)$ and $\Gamma_1(p^\infty)$-level structures in the context of the Bruhat-Tits tree or its suitable coverings, then it may allow us to consider a direct connection of the Bruhat-Tits tree and quaternionic Hida theory.
Also, it would naturally explain the relation between geometric Gross points and big Gross points \`{a} la Longo-Vigni \cite[$\S$7]{longo-vigni-manuscripta} in the ordinary case.

\subsection{Explicit computation}
One may implement an explicit overconvergent algorithm, which is expected to be as effective as one in \cite{guitart-masdeu}, to improve the computation of theta elements given in \cite[$\S$5.1]{bertolini-darmon-mumford-tate-1996}. The explicit computation of the quotient graph \cite{computing-graphs} seems helpful to do this at least for the case of weight two forms. In \cite{graf_quaternionic_L-invariant}, Peter Mathias Gr{\"{a}}f explicitly computed Teitelbaum $\mathcal{L}$-invariants of $p$-newforms. If one implement the computation of theta elements for not only $p$-newforms but also $p$-stabilized newforms, then it would have may arithmetic applications. For the cyclotomic case, see \cite{explicit-hida}.

\section*{Acknowledgement}
The author thanks Karl Rubin for a year-long discussion on this project. This project has started when the author was a visiting assistant professor at UC Irvine.
The author thanks Robert Pollack and Glenn Stevens for the inspiration of this work; Lawrence Yong-uck Chung for helping me to understand Proposition \ref{prop:orbit}; Ming-Lun Hsieh for pointing out various technical issues; Chol Park and Sug Woo Shin for helpful discussion. Some part of this work is inspired by the talk by Carlos de Vera Piquero at CRM, March 2015.
\bibliographystyle{amsalpha}
\bibliography{library}
\end{document}